\documentclass{amsart}
\usepackage{amsmath}
\usepackage{amssymb}
\usepackage{enumerate}
\usepackage{amsthm, amsfonts}
\usepackage{graphicx}
\usepackage{subfigure}
\usepackage{color}
\usepackage{mathrsfs}
\usepackage{lscape}
\usepackage{url}
\usepackage{multirow}
\usepackage{enumerate}

\newcommand{\nn}{\nonumber}
\newtheorem{theorem}{Theorem}[section]

\newtheorem{algorithm}[theorem]{Algorithm}
\newtheorem{proposition}[theorem]{Proposition}
\newtheorem{definition}[theorem]{Definition}
\newtheorem{remark}[theorem]{Remark}
\newcommand{\be}{\begin{equation}}
\newcommand{\ee}{\end{equation}}
\newcommand{\ba}{\begin{array}}
\newcommand{\ea}{\end{array}}
\newcommand{\bea}{\begin{eqnarray}}
\newcommand{\eea}{\end{eqnarray}}
\newcommand{\reff}[1]{(\ref{#1})}
\newcommand{\cred}[1]{{\color{red}{#1}}}

\newcommand{\bc}{\begin{center}}
\newcommand{\ec}{\end{center}}

\newcommand{\st}{\mathrm{s.t.}}
\newcommand{\beq}{\begin{equation}}
\newcommand{\eeq}{\end{equation}}

\newcommand{\mtA}{\mathcal{A}}
\newcommand{\mtB}{\mathcal{B}}

\newcommand{\mbR}{\mathbb{R}}
\newcommand{\mbC}{\mathbb{C}}

\newcommand{\rank}{\mathrm{rank}}
\newtheorem{example}{Example}[section]

\numberwithin{equation}{section}  
\numberwithin{table}{section}
\numberwithin{figure}{section}

\newcommand{\lmd}{\lambda}
\newcommand{\eps}{\epsilon}

\newcommand{\dt}{\delta}

\newcommand{\N}{\mathbb{N}}
\def\af{\alpha}

\def\rank{\mbox{rank}}
\newcommand{\prm}{\prime}
\newcommand{\baray}{\begin{array}}
\newcommand{\earay}{\end{array}}
\newcommand{\bca}{\begin{cases}}
\newcommand{\eca}{\end{cases}}
\newcommand{\bcen}{\begin{center}}
\newcommand{\ecen}{\end{center}}
\newcommand{\bbm}{\begin{bmatrix}}
\newcommand{\ebm}{\end{bmatrix}}
\newcommand{\bmx}{\begin{matrix}}
\newcommand{\emx}{\end{matrix}}
\newcommand{\bpm}{\begin{pmatrix}}
\newcommand{\epm}{\end{pmatrix}}
\newcommand{\btab}{\begin{tabular}}
\newcommand{\etab}{\end{tabular}}

\begin{document}

\title{All Real Eigenvalues of Symmetric Tensors}

\author{Chun-Feng Cui}
\address{
State Key Laboratory of Scientific and Engineering Computing,
Institute of Computational Mathematics and Scientific/Engineering Computing,
Academy of Mathematics and Systems Science,
Chinese Academy of Sciences,  P.O. Box 2719,  Beijing 100190,  P.R. China.
}  \email{cuichf@lsec.cc.ac.cn}

\author{Yu-Hong Dai}
\address{
State Key Laboratory of Scientific and Engineering Computing,
Institute of Computational Mathematics and Scientific/Engineering Computing,
Academy of Mathematics and Systems Science,
Chinese Academy of Sciences,  P.O. Box 2719,  Beijing 100190,  P.R. China.
}  \email{dyh@lsec.cc.ac.cn}

\author{Jiawang Nie}
\address{
Department of Mathematics,  University of California San Diego,  9500
Gilman Drive,  La Jolla,  California 92093,  USA.
} \email{njw@math.ucsd.edu}

\begin{abstract}
This paper studies how to compute all real eigenvalues,
associated to real eigenvectors, of a symmetric tensor.
As is well known,  the largest or smallest eigenvalue can be found by solving
a polynomial optimization problem,  while the other middle ones
cannot. We propose a new approach
for computing all real eigenvalues sequentially,  from the largest to the smallest.
It uses Jacobian semidefinite relaxations in polynomial optimization.
We show that each eigenvalue can be computed by
solving a finite hierarchy of semidefinite relaxations.
Numerical experiments are presented to show how to do this.
\end{abstract}



\maketitle

\section{Introduction}

Let $\mbR$ be the real field,  and let $m$ and $n$ be positive integers.
An $n$-dimensional tensor of order $m$ is an array indexed
by integer tuples $(i_1, \ldots, i_m)$ with $1 \leq i_j \leq n$ ($j=1, \ldots, m$).
Let $\mathtt{T}^m(\mbR^n)$ denote the space of all such real tensors.
A tensor $\mathcal{A} \in \mathtt{T}^m(\mbR^n)$ is indexed as
\[
\mathcal{A} = (\mathcal{A}_{i_1 \ldots  i_m} )_{1\leq i_1,  \ldots,  i_m \leq n}.
\]
The tensor $\mathcal{A}$ is {\it symmetric} if each entry
$\mathcal{A}_{i_1 \ldots  i_m}$ is invariant with respect to
all permutations of $(i_1, \ldots,  i_m)$. Let
$\mathtt{S}^m(\mbR^n)$ be the space of all symmetric tensors in $\mathtt{T}^m(\mbR^n)$.
For $\mathcal{A} \in \mathtt{S}^m(\mbR^n)$,  we denote the polynomial
\[
 \mathcal{A} x^m :=\sum_{1\leq i_1,  \ldots,  i_m \leq n}
 \mathcal{A}_{i_1  \ldots  i_m}x_{i_1}\cdots x_{i_m}.
 \]
Clearly,  $\mathcal{A} x^m$ is a form (i.e.,  a homogenous polynomial) of degree $m$
in $x:=(x_1, \ldots,  x_n)$.
For a positive integer $k \leq m$,  denote
\[
x^{[k]} := ((x_1)^k,  \ldots,  (x_n)^k).
\]
Define $\mathcal{A} x^k$ to be the symmetric tensor
in $\mathtt{S}^{m-k}(\mbR^n)$ such that
\[
(\mathcal{A} x^k)_{i_1, \ldots,  i_{m-k} } :=
\sum_{ 1 \leq j_1,  \ldots,  j_k \leq n}
\mathcal{A}_{ i_1 \ldots  i_{m-k}  j_1  \ldots  j_k } x_{j_1} \cdots x_{j_k}.
\]
So,  $\mathcal{A} x^{m-1}$ is an $n$-dimensional vector.

An important property of symmetric tensors is their eigenvalues.
Eigenvalues of tensors are introduced in Qi~\cite{Qi05} and Lim~\cite{Lim05}.
Unlike matrices,  there are various definitions of eigenvalues for tensors.
Useful ones include H-eigenvalues,  Z-eigenvalues~(cf.~\cite{Qi05}),
and D-eigenvalues~(cf.~\cite{Qi08}).  Eigenvalues of symmetric tensors have applications
in signal processing~(cf.~\cite{Qi03}),
diffusion tensor imaging (DTI)~(cf.~\cite{chen2013positive,Qi08,Qi10}),
%
%
automatic control~(cf. \cite{ni2008eigenvalue}), etc.
The tensor eigenvalue problem is an important
subject of multi-linear algebra. We refer to~\cite{Kolda2009, Lim13, QiSunWang07}
for introductions to tensors and their applications.

Since there are various definitions of eigenvalues,
we here give a unified approach to define them.
It is a variation of the approach
introduced in~\cite{chang2009eigenvalue,Lim05,Qi05}.
Let $\mbC$ be the complex field.

\begin{definition}
Let $\mtA \in \mathtt{S}^m(\mbR^n)$ and $\mtB \in \mathtt{S}^{m^\prm}(\mbR^n)$
be two symmetric tensors (their orders $m,  m^\prm$ are not necessarily equal).
A number $\lambda\in \mbC$ is a $\mathcal{B}$-eigenvalue of $\mathcal{A}$
if there exists $ u \in \mbC^n$ such that
\be \label{Beig}
\mtA u^{m-1}=\lambda\,\mtB u^{m^\prm -1},  \quad \mtB u^{m^\prm} = 1.
\ee
Such $u$ is called a $\mtB$-eigenvector associated to $\lmd$,
and such $(\lmd,  u)$ is called a $\mathcal{B}$-eigenpair of $\mathcal{A}$.
\end{definition}

For cleanness of the paper,  when the tensor $\mtB$ is clear in the context,
$\mtB$-eigenvalues (resp.,  $\mtB$-eigenvectors,  $\mtB$-eigenpairs)
are just simply called eigenvalues (resp.,  eigenvectors,  eigenpairs).
When an eigenvalue $\lmd$ is real,  it may not have a real
eigenvector $u$.
An eigenpair $(\lmd,u)$ is called {\it real} if
both $\lmd$ and $u$ are real.
Throughout the paper, for convenience,
we call that $\lmd$ is a real eigenvalue if
$\lmd$ is real and it has a real eigenvector.
By the largest (resp.,  smallest) eigenvalue,
we mean the largest (resp.,  smallest) real eigenvalue.
In the paper,  we only discuss how to compute real eigenvalues.

%
%

The following special cases of $\mathcal{B}$-eigenvalues are well known.

\begin{itemize}
%
%

\item  When $m^\prm=m$ and $\mathcal{B}$ is the identity tensor (i.e.,
$\mathcal{B}x^m = x_1^m+\cdots+x_n^m$), the $\mtB$-eigenvalues are
just the H-eigenvalues (cf.~\cite{Qi05}). When $m$ is even,
a number $\lambda$ is a real H-eigenvalue of $\mathcal{A}$
if there exists $u \in \mbR^n$ such that
\be \label{Heig}
\nn \mtA u^{m-1} = \lambda\,u^{[m-1]},  \quad u_1^m + \cdots + u_n^m = 1.
\ee
Such $(\lmd,  u)$ is called an H-eigenpair.

\item When $m^\prm =2$ and $\mathcal{B}$ is such that
$\mathcal{B}x^2 = x_1^2+\cdots+x_n^2$,
the $\mtB$-eigenvalues are just the Z-eigenvalues (cf.~\cite{Qi05}).
Equivalently,  a number $\lambda$ is a real Z-eigenvalue
if there exists $u\in \mbR^n$ such that
\be \label{Zeig}
\nn \mtA u^{m-1} = \lambda\,u, \quad u_1^2+\cdots+u_n^2=1.
\ee
Such $(\lmd,  u)$ is called a Z-eigenpair.

\item Let $D \in \mbR^{n\times n}$ be a symmetric positive definite matrix.
When $m^\prm = 2$ and $\mtB$ is such that
$\mathcal{B}x^2 = x^{\text{T}}Dx$,  the $\mtB$-eigenvalues are just the
D-eigenvalues (cf.~\cite{Qi08}). Equivalently,  a number $\lmd$ is a real D-eigenvalue
if there exists $u\in \mbR^n$ such that
\be \label{Deig}
\nn \mtA u^{m-1} = \lambda\,Du, \quad u^{\text{T}}Du=1.
\ee
Such $(\lmd,  u)$ is called a D-eigenpair.

\end{itemize}

%
%

The problem of computing eigenvalues of higher order tensors (i.e.,  $m \geq 3$)
is NP-hard (cf.~\cite{hillar2009most}).
Recently,  there exists much work
for computing the largest (or smallest) eigenvalues of symmetric tensors.
Qi et al.~\cite{Qi09} proposed an elimination method
for computing the largest Z-eigenvalue when $n = 2$ and $m=3$.
Hu et al.~\cite{HHQ13} used a sequence of semidefinite relaxations
for computing extreme Z-eigenvalues.
Kolda et al.~\cite{SSHOPM} presented a shifted power method for computing
Z-eigenvalues. Zhang et al.~\cite{ZLQ12} proposed a modified power method.
Han~\cite{Han12} introduced an unconstrained optimization method
for even order symmetric tensors.
Hao et al.~\cite{Hao12} presented a sequential subspace projection method
for computing extreme Z-eigenvalues.

The existing methods are mostly for computing the largest or smallest eigenvalues.
However,  there are very few methods for computing the other
middle eigenvalues. Computing the second or other largest eigenvalues for
symmetric tensors is also an important problem in some applications.
In DTI~\cite{chen2013positive,Qi10}, the three largest Z-eigenvalues
of a diffusion tensor describe the diffusion coefficients in different directions.
As shown by Li et al.~\cite{li2013z},  the second largest Z-eigenvalue
for the characteristic tensor of a hypergraph can be used to get a lower bound
for its bipartition width.
%
%

The main goal of this paper is to compute all
real eigenvalues of a symmetric tensor. For
$\mtA \in \mathtt{S}^m(\mbR^n)$,  $\mtB \in \mathtt{S}^{m^\prm}(\mbR^n)$,  it holds that
\[
\nabla\mtA x^m = m\,\mtA x^{m-1},  \quad
\nabla\mtB x^{m^\prm} = m^{\prm}\, \mtB  x^{m^\prm-1}.
\]
Here, the symbol $\nabla$ denotes the gradient in $x$.
Thus,  \reff{Beig} is equivalent to
\[
\frac{1}{m}\, \nabla\mtA u^m  = \frac{1}{m^\prm}\, \lmd\, \nabla\mtB u^{m^\prm},  \quad \mtB u^{m^\prm} = 1.
\]
%
%
Then,  $(\lmd,  u)$ is a $\mtB$-eigenpair if and only if
$u$ is a critical point of the problem
\be \label{Bmodel}
 \max \quad  \mathcal{A}x^m  \quad \st  \quad \mtB x^{m^\prm} = 1.
\ee
Moreover,  the critical value associated to $u$ is $\lmd$,  because
\[
u^{\text{T}} \nabla\mtA u^m  = m\, \mtA u^m,  \quad
u^{\text{T}} \nabla\mtB u^{m^\prm}= m^\prm\, \mtB u^{m^\prm}.
\]
This shows that $(\lmd,  u)$ is a $\mtB$-eigenpair if and only if
$u$ is a critical point of \reff{Bmodel} with the critical value $\lmd$.
The polynomial optimization problem \reff{Bmodel}
has finitely many critical values (cf.~\cite{nie2013hierarchy}),
including both complex and real ones.
That is,  every symmetric tensor $\mtA$ has finitely many
complex and real $\mtB$-eigenvalues.
We order  the real $\mtB$-eigenvalues  monotonically as
$\lmd_1 > \lmd_2 > \cdots > \lmd_K$.
For convenience,  denote $\lmd_{\max} := \lmd_1$ and $\lmd_{\min} := \lmd_K$.

In this paper,  we study how to compute all real eigenvalues.
Mathematically,  this is equivalent to
finding all the real critical values of \reff{Bmodel},
which is a polynomial optimization problem.
The semidefinite relaxation method by Lasserre~\cite{lasserre2001global}
can be applied to get the largest or smallest eigenvalue.
To get other middle eigenvalues,  we need to use new techniques.
Recently,  Nie~\cite{nie2013hierarchy} proposed a method for computing
the hierarchy of local minimums in polynomial optimization,
which uses the Jacobian SDP relaxation method from~\cite{nie2013exact}.
We mainly follow the approach in~\cite{nie2013hierarchy}
to compute all real eigenvalues sequentially.
Indeed,  by this approach,  each real eigenvalue can be obtained
by solving a finite hierarchy of semidefinite relaxations.
This is an attractive property that most other numerical methods do not have.

The paper is organized as follows. In \S \ref{section-preliminaries},
we present some preliminaries in polynomial optimization.
In \S \ref{section_Beigen},  we propose semidefinite relaxations
for computing all real eigenvalues sequentially.
In \S \ref{section_numerical}, we report extensive numerical examples
to show how to compute all real eigenvalues.

\section{Preliminaries}\label{section-preliminaries}

This section reviews some basics in polynomial optimization.
We refer to~\cite{cox2007ideals, LasBok, Lau} for details.

Denote by $\mbR[x]:=\mbR[x_1,  \ldots,  x_n]$ the ring of polynomials
in $x:=(x_1, \ldots,  x_n)$  with real coefficients.
For a degree $d$,  $\mbR[x]_d$ denotes the space of all polynomials in $\mbR[x]$
whose degrees are at most $d$. The dimension of the space
$\mbR[x]_d$ is $\binom{n+d}{d}$.
An ideal of $\mbR[x]$ is a subset $J$ of $\mbR[x]$
such that $J \cdot \mbR[x] \subseteq J$ and $J+J \subseteq J$.
For a tuple $\phi:=(\phi_1, \ldots, \phi_r)$ of polynomials in $\mbR[x]$,
the ideal generated by $\phi$ is the smallest ideal containing all $\phi_i$,
which is the set $\phi_1\cdot \mbR[x]+\cdots+\phi_r\cdot \mbR[x]$
and is denoted by $I(\phi)$. The set
\[
I_k(\phi) := \phi_1\cdot \mbR[x]_{k-\deg(\phi_1)} + \cdots
+ \phi_r\cdot \mbR[x]_{k-\deg(\phi_r)}
\]
is called the $k$-th truncation of the ideal $I(\phi)$. Clearly,
\[
\bigcup_{k \in \N } \,  I_k(\phi) = I(\phi).
\]

A polynomial $\sigma \in \mbR[x]$ is called a sum of squares (SOS)
if there exist $p_1,  \ldots,  p_k \in \mbR[x]$ such that
$\sigma = p_1^2 + \cdots +p_k^2$. Let $\Sigma[x]$ be the set of all SOS polynomials and
\[
\Sigma[x]_m := \Sigma[x] \cap \mbR[x]_m.
\]
Both $\Sigma[x]$ and $\Sigma[x]_m$ are convex cones. As is well known,
each SOS polynomial is nonnegative everywhere,
while the reverse is not necessarily true. We refer to~\cite{Rez00}
for a survey on SOS and nonnegative polynomials.
Let $\psi:=(\psi_1,  \ldots,  \psi_t)$ be a tuple of polynomials in $\mbR[x]$. The set
\[
Q_N(\psi) := \Sigma[x]_{2N} + \psi_1\cdot \Sigma[x]_{2N-\deg(\psi_1)} + \cdots
+ \psi_t\cdot\Sigma[x]_{2N-\deg(\psi_t)}
\]
is called the $N$-th truncation of the quadratic module
generated by $\psi$. The union
\[
Q(\psi) := \bigcup_{N \in \N } \,  Q_N(\psi)
\]
is called the quadratic module generated by $\psi$.

Let $\N$ be the set of nonnegative integers.
For $x:=(x_1, \ldots,  x_n)$ and $\af: = (\af_1,  \ldots,  \af_n)$,  denote
$x^\af := x_1^{\af_1} \cdots x_n^{\af_n}$ and
$|\af| := \af_1 + \cdots + \af_n$. For $d\in \N$,  denote
\[
\N_d^n := \{ \af \in \N^n: |\af| \leq d\}.
\]
The space dual to $\mbR[x]_d$ is the set of all
{\it truncated multi-sequences} (tms') of degree $d$,
which is denoted by $\mbR^{\N_d^n}$. A vector $y$ in $\mbR^{\N_d^n}$
is indexed by $\af \in \N_d^n$,  i.e.,
\[
y = (y_\af)_{ \af \in \N_d^n}.
\]
Each $y \in \mbR^{\N_d^n}$ defines the linear functional $\mathscr{L}_y $
acting on $\mbR[x]_d$ as
\[
\mathscr{L}_y \left(  x^\af \right) =  y_\af, \quad \forall \,  \af \in \N_d^n.
\]

Let $q \in \mbR[x]_{2k}$. For each $y \in \mbR^{\N_{2k}^n}$,
the function $\mathscr{L}_y (q p^2)$
is a quadratic form in $vec(p)$,  the coefficient vector of
polynomial $p$ with $\deg(qp^2) \leq 2k$.
Let $L_q^{(k)}(y)$ be the symmetric matrix such that
\[
\mathscr{L}_y ( q p^2  ) =  vec(p)^{\text{T}} \Big( L_q^{(k)}(y) \Big) vec(p).
\]
The matrix $L_q^{(k)}(y)$ is called the $k$-th localizing matrix of $q$
generated by $y$. It is linear in $y$.
For instance,  when $n=2$,  $k=2$ and $q=1-x_1^2-x_2^2$,  we have
\[
L_{1-x_1^2-x_2^2}^{(2)}(y) =  \left(
\begin{array}{rcc}
y_{00}-y_{20}-y_{02} & y_{10}-y_{30}-y_{12} & y_{01}-y_{21}-y_{03} \\
y_{10}-y_{30}-y_{12} & y_{20}-y_{40}-y_{22} & y_{11}-y_{31}-y_{13} \\
y_{01}-y_{21}-y_{03} & y_{11}-y_{31}-y_{13} & y_{02}-y_{22}-y_{04} \\
\end{array}
\right).
\]
When $q=1$ (i.e.,  the constant one polynomial),
$L_1^{(k)}(y)$ is called the $k$-th moment matrix generated by $y$,
and is denoted as $M_k(y)$. For instance,  when $n=2$ and $k=2$,
 \[
 M_2(y)= \left(
\begin{array}{cccccc}
      y_{00} & y_{10} & y_{01} & y_{20} & y_{11} & y_{02} \\
      y_{10} & y_{20} & y_{11} & y_{30} & y_{21} & y_{12} \\
      y_{01} & y_{11} & y_{02} & y_{21} & y_{12} & y_{03} \\
      y_{20} & y_{30} & y_{21} & y_{40} & y_{31} & y_{22} \\
      y_{11} & y_{21} & y_{12} & y_{31} & y_{22} & y_{13} \\
      y_{02} & y_{12} & y_{03} & y_{22} & y_{13} & y_{04} \\
\end{array}
  \right).
\]

%
%
%

\section{Semidefinite relaxations for computing all real eigenvalues}
\label{section_Beigen}

In this section,  we show how to compute all real eigenvalues sequentially.
The Jacobian SDP relaxation technique in~\cite{nie2013exact}
is a useful tool for this purpose.

Let $\mtA \in \mathtt{S}^m(\mbR^n)$ and $\mtB \in \mathtt{S}^{m^\prm}(\mbR^n)$.
For convenience,  denote $f(x):= \mtA x^m$ and $g(x):= \mtB x^{m^\prm}-1$.
Then \reff{Bmodel} is the same as
\be \label{equ:maxf-st-g}
 \max \quad f(x) \quad \st \quad  g(x)=0.
\ee
%
%
In the introduction,  we have seen that $(\lmd,  u)$ is a $\mtB$-eigenpair
if and only if $\lmd$ is a critical value of \reff{equ:maxf-st-g},
and $u$ is an associated critical point.
The problem~\reff{equ:maxf-st-g} always has
finitely many critical values (cf.~\cite{nie2013hierarchy}),
including both complex and real ones.
So, $\mtA$ has finitely many complex and real eigenvalues.
We order the real eigenvalues monotonically as
\[
\lmd_1 > \lmd_2 > \cdots > \lmd_K,
\]
where $K$ is the total number of distinct real eigenvalues. Denote
\[
\mathcal{W} := \{x\in \mathbb{R}^n\ |\ \rank\, [\nabla f(x) \quad \nabla g(x) ]\leq 1\}.
\]
Clearly,  if $(\lmd,  u)$ is a real $\mtB$-eigenpair of $\mtA$,
then $u \in \mathcal{W}$.
The description of the set $\mathcal{W}$ does not use the Lagrange multiplier.
This is an advantage in computation.
Suppose $g(x)=0$ is a smooth real hypersurface
(i.e.,  $\nabla g(x) \ne 0$ for all real points on $g(x) = 0$).
It follows from   Definition 1.1 that any $u \in \mathcal{W}$ satisfying $g(u)=0$
is a $\mtB$-eigenvector of $\mtA$, associated to the eigenvalue $\lmd = f(u)$.
For the frequently used Z-eigenvalues (i.e.,  $g(x) = x^{\text{T}}x-1$)
and H-eigenvalues (i.e.,  $g(x) = x_1^m+\cdots+x_n^m-1$),
the hypersurface $g(x)=0$ is smooth.

A point $u$ belongs to $\mathcal{W}$ if and only if
\[
f_{x_i}(u)g_{x_j}(u)-f_{x_j}(u)g_{x_i}(u)=0 \,
\,  ( 1 \leq i < j \leq n ),
\]
where $f_{x_i}=\frac{\partial}{\partial x_i} f(x)$
and $g_{x_i}=\frac{\partial}{\partial x_i}g(x)$.
There are totally $\frac{1}{2} n(n-1)$ equations.
Indeed,  the number of defining equations for $\mathcal{W}$
can be dropped to $2n-3$ (cf.~\cite[Chap.~5]{Bruns}).
It suffices to use the following $2n-3$ equations (cf.~\cite{Bruns, nie2013exact}):
\be \label{def:hr}
h_r:=\sum_{i+j=r+2}{(f_{x_i}g_{x_j}-f_{x_j}g_{x_i}})=0 \, \,  (r=1, \cdots, 2n-3).
\ee
For convenience,  let $h_{2n-2} := g$ and
\be \label{def:h}
h := (h_1,  \ldots,  h_{2n-2} ).
\ee
Clearly, \reff{equ:maxf-st-g} is equivalent to the maximization problem
\be \label{align_general_new}
\max \quad f(x) \quad \st \quad  h_r(x)=0 \, \,  ( r = 1,  \ldots,  2n-2 ).
\ee
When the real hypersurface $g(x)=0$ is smooth,
a point $u$ is feasible for \reff{align_general_new}
 if and only if $u$ is a critical point of \reff{equ:maxf-st-g},
i.e.,  $u$ is a $\mtB$-eigenvector. This implies that
the objective value of \reff{align_general_new} at any feasible point
is a $\mtB$-eigenvalue of $\mtA$. Thus, 
the objective values on feasible points
are $\lmd_1,  \ldots,  \lmd_K$.

In the following,  we show how to compute all real eigenvalues sequentially.
That is,  we compute $\lmd_1$ first,  then $\lmd_2$ second,
and then $\lmd_3,  \ldots$ if they exist.

\subsection{The largest eigenvalue}

The largest eigenvalue $\lmd_1$ is the maximum value
of problem \reff{align_general_new}.
Write the polynomial $f(x)= \mathcal{A} x^m$ as
\[
f(x) = \sum_{\af \in \N^n: |\alpha| = m}  {f_{\alpha} x^{\alpha}}.
\]
For a tms $y \in \mbR^{\N_{2N}^n}$ with degree $ 2N \geq m$,  denote
\[
\langle f, y\rangle := \sum_{\af \in \N^n: |\alpha| = m}  {f_{\alpha}y_{\alpha}}.
\]
Clearly,  $\langle f, y\rangle$ is a linear function in $y$.
Denote
\[ N_0 := \lceil (m+m^\prm-2)/2\rceil. \]
Lasserre's hierarchy of semidefinite relaxations (cf.~\cite{lasserre2001global})
for solving \reff{align_general_new} is ($N=N_0, N_0+1, \ldots$)
\be \label{equ_hierarchy_p1}
\left\{ \baray{rl}
 \rho_N^{(1)}:=\max & \langle f, y\rangle  \\
  \st & L_{h_r}^{(N)}(y) =0\, \,  ( r=1, \cdots, 2n-2),  \\
   &  y_0 = 1, \,  M_N(y)\succeq 0.
\earay \right.
\ee
Let $h$ be the tuple as in \reff{def:h}.
The dual problem of \reff{equ_hierarchy_p1} is then
\be \label{largest_sos}
\eta_N^{(1)} := \min \quad \gamma \quad
\st\quad \gamma- f \in I_{2N}(h) + \Sigma[x]_{2N}.
\ee

It can be shown that the optimal values
$\rho_N^{(1)},  \eta_N^{(1)}$ are upper bounds for $\lmd_1$.
Both sequences $\{\rho_N^{(1)} \}$
and $\{ \eta_N^{(1)} \}$ are monotonically decreasing.
%
%
That is
\[
\rho_{N_0}^{(1)}  \geq \rho_{N_0+1}^{(1)} \geq \cdots \geq  \rho_N^{(1)} \geq \cdots  \geq \lmd_1,
\]
\[
\eta_{N_0}^{(1)}  \geq \eta_{N_0+1}^{(1)} \geq \cdots \geq  \eta_N^{(1)} \geq \cdots  \geq \lmd_1.
\]
By the weak duality,  we also have
\[
\rho_N^{(1)}  \leq  \eta_N^{(1)}  \quad  (N=N_0, N_0+1, \ldots).
\]
In fact,  they both have the nice property
of converging to $\lmd_1$ in finitely many steps,
i.e., $\rho_N^{(1)} = \eta_N^{(1)}=\lambda_1$ for all $N$ large enough.

\begin{theorem} \label{thm:lmd:1st}
Let $\mtA \in \mathtt{S}^m(\mbR^n)$ and $\mtB \in \mathtt{S}^{m^\prm}(\mbR^n)$.
Suppose the real hypersurface $\mathcal{B}x^{m^\prm} =1$ is smooth.
Let $\lambda_1$ be the largest real $\mtB$-eigenvalue of $\mtA$.
Then,  we have the following properties:
\begin{enumerate}[(i)]

\item It holds that
$\rho_N^{(1)} = \eta_N^{(1)}=\lambda_1$ for all $N$ large enough.

\item Suppose $\lmd_1$ has finitely many real eigenvectors on $\mathcal{B}x^{m^\prm} =1$.
If $N$ is large enough,
then,  for every optimizer $y^*$ of \reff{equ_hierarchy_p1},
there exists an integer $t \leq N$ such that
\be \label{cd:flat}
\rank\,  M_{t-N_0}(y^*) \,  = \,  \rank \,  M_{t}(y^*).
\ee
\end{enumerate}
\end{theorem}

\begin{proof}
Note that $-\lmd_1$ is the minimum value of
\[
\min \quad -f(x) \quad \st \quad g(x) = 0.
\]
The polynomials $h_1,  \ldots,  h_{2n-3}$ are constructed
by using Jacobian SDP relaxations in~\cite{nie2013exact}.
The relaxations \reff{def:hr}, \reff{align_general_new}, \reff{equ_hierarchy_p1}-\reff{largest_sos}
are specializations of the semidefinite relaxations (4.5), (4.6),  (4.7)-(4.8)
constructed in~\cite{nie2013hierarchy}.
Thus, the items (i)-(ii) can be implied by
Theorem 4.1 of~\cite{nie2013hierarchy}.
\end{proof}

In computation,  a practical issue is how to determine whether
$\rho_N^{(1)} = \eta_N^{(1)}=\lambda_1$ or not, because $\lmd_1$ is typically unknown.
This can be done by checking the rank condition \reff{cd:flat}.
If it is satisfied,  then we can get
\[
\ell:=\rank \,  M_{t}(y^*)
\]
distinct feasible points $u_1,  \ldots,  u_\ell$ of \reff{align_general_new},
such that each $u_i$ is a maximizer of \reff{align_general_new}
and $f(u_i) = \rho_{N}^{(1)} = \eta_{N}^{(1)} =\lmd_1$.
They can be computed by the method in Henrion and Lasserre~\cite{HenLas05}.
In other words,  if \reff{cd:flat} holds,  then
$\rho_{N}^{(1)} = \eta_{N}^{(1)} = \lmd_1$,  and such
$u_1,  \ldots,  u_\ell$ are the associated $\mtB$-eigenvectors.
So,  by solving  \reff{equ_hierarchy_p1}-\reff{largest_sos},  we not only
get the largest eigenvalue $\lmd_1$,  but also its $\mtB$-eigenvectors.
As shown in Theorem~3.1 (ii), if there are finitely many real $\mtB$-eigenvectors
(this is the general case, cf.~\cite{cartwright2011number}),
then \reff{cd:flat} must be satisfied.
So, \reff{cd:flat} is generally sufficient and necessary for checking
convergence of semidefinite relaxations \reff{equ_hierarchy_p1}-\reff{largest_sos}.
The rank condition \reff{cd:flat} is called flatness.
It is a very useful tool for solving truncated moment problems
(cf. Curto and Fialkow~\cite{CF05}).
The software {\tt GloptiPoly~3} (cf. ~\cite{GloPol3}) can be applied to solve
the semidefinite relaxations \reff{equ_hierarchy_p1}-\reff{largest_sos}.

In Theorem~\ref{thm:lmd:1st}, the relaxations \reff{equ_hierarchy_p1}-\reff{largest_sos}
are assumed to be solved exactly. However, in practice,
they are often solved approximately, due to round-off errors.
Suppose $\widetilde{\rho}_{N}^{(1)},\widetilde{\eta}_{N}^{(1)}$
are numerically computed optimal values of \reff{equ_hierarchy_p1}-\reff{largest_sos},
respectively. Then
$\widetilde{\rho}_{N}^{(1)} = \widetilde{\eta}_{N}^{(1)}= \lambda_1$
may not hold exactly, but they are approximately true. The errors depend
on the accuracy of solving \reff{equ_hierarchy_p1}-\reff{largest_sos}.
We refer to Chapter~7 of the book \cite{SDPbook}
for error analysis in semidefinite programming.
When approximately optimal solutions of \reff{equ_hierarchy_p1}-\reff{largest_sos}
are computed, the rank condition \reff{cd:flat} will be satisfied  approximately.
This issue was discussed in \cite[\S3]{Nie-ft}.

\begin{remark}\label{rmk:geomul}
Suppose the rank condition \reff{cd:flat} is satisfied.
If $\rank\,  M_N(y^*)$ is maximum among the set of all optimizers of \reff{equ_hierarchy_p1},
then we can get all maximizers of \reff{align_general_new}
(cf.~\cite[\S6.6]{Lau}).
In such case,  we can get all the $\mtB$-eigenvectors associated to $\lmd_1$.
Therefore,  when \reff{equ_hierarchy_p1}-\reff{largest_sos} are solved by
primal-dual interior point methods, typically we can get
all the $\mtB$-eigenvectors associated to $\lmd_1$ (cf.~\cite{nie2013hierarchy}).
However,  if there are infinitely many $\mtB$-eigenvectors lying on $\mtB x^{m^\prm} =1$,
\reff{cd:flat} is typically not satisfied.
\end{remark}

To check the condition \reff{cd:flat},
we need to evaluate the ranks of matrices $M_{t-N_0}(y^*)$ and $M_{t}(y^*)$.
In numerical computation, sometimes this would be a very difficult issue
because of round-off errors. The rank of a matrix equals to the number of its
positive singular values. In practice, we can evaluate the rank
as the number of singular values bigger than a tolerance (say, $10^{-6}$).
By this way, if there is a sufficiently small perturbation on a matrix,
its evaluated rank will not change.
We refer to the book \cite{Demmel} for evaluating matrix ranks numerically.

\subsection{The second and other largest eigenvalues}

Suppose the $k$-th largest eigenvalue
$\lambda_k$ of $\mathcal{A}$ is known. We want to compute the
$(k+1)$-th largest eigenvalue $\lambda_{k+1}$,  if it exists.
Let $\delta\in \mathbb{R}$ be such that
\be \label{equ_delta}
 0 < \delta < \lmd_k - \lmd_{k+1}.
\ee
Consider the optimization problem
\be \label{largest-r-Jacobian}
\left\{ \baray{rl}
\max &  f(x)  \\
  \st  & h_r(x)=0 \, \,  ( r=1, \ldots, 2n-2),     \\
       &  f(x) \le \lambda_k - \delta.
\earay \right.
\ee
When \reff{equ_delta} is satisfied,  the optimal value
of \reff{largest-r-Jacobian} is $\lmd_{k+1}$.
Lasserre's hierarchy of semidefinite relaxations for solving
\reff{largest-r-Jacobian} is ($N=N_0, N_0+1, \ldots$)
\be \label{equ_hierarchy_pk}
\left\{ \baray{rl}
 \rho_N^{(k+1)} := \max & \langle f, y\rangle \\
\st  & L_{h_r}^{(N)}(y) =0 \,  (r=1, \ldots, 2n-2),  \\
   &  y_0 =1, \,  L_{\lambda_k-\delta-f}^{(N)}(y)\succeq 0, \,  M_N(y)\succeq 0.
\earay \right.
\ee
Its dual problem is then
\be   \label{largest_r_sos}
\eta_N^{(k+1)} := \min \quad \gamma \quad
\st \quad \gamma- f \in I_{2N}(h) + Q_N(\lambda_k-\delta-f).
\ee

Semidefinite relaxations \reff{equ_hierarchy_pk}-\reff{largest_r_sos}
have the following properties:

\begin{theorem} \label{thm:lmd:k+1}
Suppose the real hypersurface $\mtB x^{m^\prm}=1$ is smooth.
Let $\lmd_k$ (resp.,  $\lmd_{k+1}$) be the $k$-th (resp.,  $(k+1)$-th )
largest $\mathcal{B}$-eigenvalue of $\mtA$.
For all $\delta$ satisfying \reff{equ_delta},
we have the following properties:
\begin{enumerate}[(i)]
%
%
\item For all $N$ big enough,  we have
$\rho_N^{(k+1)} = \eta_N^{(k+1)}=\lambda_{k+1}$.

\item Suppose $\lmd_{k+1}$  has finitely many eigenvectors on
$\mathcal{B}x^{m^\prm} =1$. If $N$ is large enough,  then
for every optimizer $y^*$ of \reff{equ_hierarchy_pk},
there exists an integer $t \leq N$ such that \reff{cd:flat} holds.

\end{enumerate}
\end{theorem}
\begin{proof}
%
%
Note that $-\lmd_k$ is the $k$-th smallest critical value of
\[
\min \quad -f(x) \quad  \text{s.t.} \quad g(x) = 0.
\]
The polynomials $h_1, \ldots,  h_{2n-3}$ are constructed by using
Jacobian SDP relaxations in~\cite{nie2013exact}.
The semidefinite relaxations \reff{largest-r-Jacobian}-\reff{largest_r_sos}
are specializations of (4.9)-(4.11) in~\cite{nie2013hierarchy}.
Thus,  the items (i)-(ii) can be obtained by Theorem 4.3 of~\cite{nie2013hierarchy}.
\end{proof}

\begin{remark} \label{nb:egv:lmdk}
The finite convergence of $\rho_N^{(k+1)}$ and $\eta_N^{(k+1)}$ to $\lmd_{k+1}$
can be identified by checking the rank condition \reff{cd:flat}.
If it is satisfied,  we can get $\ell$
$\mtB$-eigenvectors associated to $\lmd_{k+1}$.
When the semidefinite relaxations \reff{equ_hierarchy_pk} and \reff{largest_r_sos}
are solved by primal-dual interior point methods,
typically we can get all $\mtB$-eigenvectors,
provided there are finitely many ones.
%
%
The rank condition \reff{cd:flat} is generally sufficient
and necessary for checking the finite convergence
of the sequences $\{\rho_{N}^{(k+1)}\}$ and $\{\eta_{N}^{(k+1)}\}$.
We refer to Remark~\ref{rmk:geomul}.
We also refer to the discussions before and after Remark~\ref{rmk:geomul},
about the numerical issues related to \reff{cd:flat},
\reff{equ_hierarchy_pk} and \reff{largest_r_sos}.
\end{remark}

In practice,  we usually do not know whether $\lmd_{k+1}$ exists or not.
Even if it exists,  we do not know how small $\dt$ should be
chosen to satisfy \reff{equ_delta}. Interestingly,
this issue can be fixed by solving the optimization problem
\be  \label{opt:chi:k}
\left\{ \baray{rl}
\chi_k := \min & f(x) \\
  \st &  h_r(x)=0 \,  ( r=1, \ldots,  2n-2),   \\
& f(x) \ge \lambda_k - \delta.
\earay \right.
\ee
The following proposition is useful.

\begin{proposition} \label{prop:3.5}
Suppose the real hypersurface $\mtB x^{m^\prm}=1$ is smooth.
Let $\lmd_k$ (resp.,  $\lmd_{min}$) be the $k$-th largest (resp.,  smallest)
$\mathcal{B}$-eigenvalue of $\mtA$.
For all $\delta >0$,  we have the following properties:
\begin{enumerate}[(i)]

\item The relaxation \reff{equ_hierarchy_pk} is infeasible for some $N$
if and only if $\lmd_k - \dt < \lmd_{min}$.

\item If $\chi_k = \lmd_k$ and $\lmd_{k+1}$ exists,
then $\lmd_{k+1} < \lmd_k - \dt$,  i.e.,  \reff{equ_delta} holds.

\item If $\chi_k = \lmd_k$ and \reff{equ_hierarchy_pk} is infeasible for some $N$,
then $\lmd_k = \lmd_{min}$ and $\lmd_{k+1}$ does not exist.

\end{enumerate}
\end{proposition}
\begin{proof}
(i) This can be implied by Theorem 4.3 (i) of~\cite{nie2013hierarchy}.

(ii) Clearly,  $\chi_k$ is the smallest $\mtB$-eigenvalue
greater than or equal to $\lmd_k-\dt$.
If $\lmd_{k+1}$ exists and $\chi_k = \lmd_k$,  we must have
$\lmd_{k+1} < \lmd_k - \dt$.

(iii) From (i),  we know $\lmd_k - \dt < \lmd_{min}$.
If otherwise  $\lmd_{min} < \lmd_k$,  then $\lmd_{k+1}$ exists
and $\lmd_{k+1} < \lmd_k - \dt$ by (ii).
This results in the contradiction $\lmd_{k+1} < \lmd_{min}$.
So,  $\lmd_{min} = \lmd_k$.
\end{proof}

%
%
The problem \reff{opt:chi:k} is also a polynomial optimization problem.
Similar semidefinite relaxations like \reff{equ_hierarchy_pk}-\reff{largest_r_sos}
can be constructed to solve it. The hierarchy of such relaxations
can also be shown to have finite convergence (cf.~\cite{nie2013hierarchy}),
by similar arguments.
Thus, the optimal value $\chi_k$ of \reff{opt:chi:k} can be computed
by solving its semidefinite relaxations.
For $\dt>0$ sufficiently small, we must have
$\chi_k = \lmd_k$, no matter if $\lmd_{k+1}$ exists or not.
This is because $\chi_k$ is the smallest $\mtB$-eigenvalue
greater than or equal to $\lmd_k-\dt$.

The existence of $\lmd_{k+1}$ and the relation \reff{equ_delta}
can be checked as follows. First, we choose a small value
(say, 0.05) for $\dt$, and then solve \reff{opt:chi:k}.
If $\chi_k < \lmd_k$, we decrease the value $\dt$ as $\dt := \dt/5$
and solve \reff{opt:chi:k} again. Repeat this process
until we get $\chi_k = \lmd_k$.
(This process must stop when $\dt>0$ is sufficiently small.)
After $\chi_k = \lmd_k$ is reached,
there are only two possibilities:
1) If $\lmd_{k+1}$ does not exist, then $\lmd_k = \lmd_{min}$.
By Proposition~\ref{prop:3.5}(i), the relaxation \reff{equ_hierarchy_pk}
must be infeasible for some $N$. This then confirms the nonexistence
of $\lmd_{k+1}$, by Proposition~\ref{prop:3.5}(iii).
2) If $\lmd_{k+1}$ exists, then $\lmd_{k+1} < \lmd_k - \dt$,
by Proposition~\ref{prop:3.5}(ii). So, \reff{equ_delta} is satisfied.
Then, by Theorem~\ref{thm:lmd:k+1}(i),
we have $\rho_N^{(k+1)}=\lmd_{k+1}$ for $N$ sufficiently large.
In summary, if $\lmd_{k+1}$ does not exist,
we can get a certificate for that; if it exists,
we can get $\lmd_{k+1}$ by solving the relaxation \reff{equ_hierarchy_pk}.

We would like to point out that some variations of eigenvalue problems
can also be solved by using similar semidefinite relaxations.
%
%
The largest real eigenvalue in an interval $[a, b]$
is the optimal value of the problem
\be
\left\{ \baray{rl}
 \max & f(x)  \\
  \st &  h_r(x)=0 \,  (r=1, \ldots, 2n-2),   \\
 &  a \leq f(x) \leq b.
\earay \right.
\ee
If, in advance, we know there exists an eigenvector $u$ for $\lmd_{k+1}$
lying in some region, say, defined by some polynomial inequalities
$p_1(x) \geq 0, \ldots, p_s(x) \geq 0$,
then we can get such $u$ by solving the optimization problem
\be \label{eig:p(x)>=0}
\left\{ \baray{rl}
 \max & f(x)   \\
  \st & h_r(x)=0 \,  ( r=1, \ldots, 2n-2),   \\
      & f(x) \leq \lmd_k - \dt, \\
      & p_1(x) \geq 0, \, \ldots, \, p_s(x) \geq 0.
\earay \right.
\ee
Similar semidefinite relaxations like \reff{equ_hierarchy_pk}-\reff{largest_r_sos}
can be constructed to solve such polynomial optimization problems,
and we can get the desired eigenpairs.

\subsection{Getting all real eigenpairs}

We can compute all  real $\mtB$-eigenvalues sequentially as follows.
First,  we compute the largest one $\lmd_1$
by solving the hierarchy of semidefinite relaxations
\reff{equ_hierarchy_p1}-\reff{largest_sos}.
As shown in Theorem~\ref{thm:lmd:1st},  this hierarchy converges in finitely many steps.
After getting $\lmd_1$,  we solve the hierarchy of
\reff{equ_hierarchy_pk}-\reff{largest_r_sos} for $k=1$.
If $\chi_1 = \lmd_1$ and \reff{equ_hierarchy_pk} is infeasible
for some $N$,  then $\lmd_1$ is the smallest eigenvalue.
If $\chi_1 = \lmd_1$ and \reff{equ_hierarchy_pk} is feasible for all $N$,
then $\rho_N^{(2)}= \lmd_2$   for $N$ big enough.
Repeating this procedure,  we can get $\lmd_3,  \lmd_4,  \ldots$
if they exist,  or we get the smallest eigenvalue and stop.

As above,  we get the following algorithm.

\begin{algorithm} \label{alg:alleig}
Computing all   real $\mtB$-eigenpairs of a symmetric tensor $\mtA$.
\begin{description}
\item[Step 0] Choose a small positive value $\delta_0$ (e.g.,  $0.05$).
Let $k=1$.

\item [Step 1] Solve the hierarchy of \reff{equ_hierarchy_p1}
and get the largest eigenvalue $\lambda_1$.

\item [Step 2] Let $\delta = \delta_0$ and
solve the optimal value $\chi_k$ of \reff{opt:chi:k}.
If $\chi_k = \lmd_k$, then go to Step~3;
If $\chi_k < \lmd_k$,
let $\dt := \min\,( \dt/5,  \lmd_k - \chi_k )$, and compute $\chi_k$.
Repeat this process until $\chi_k = \lmd_k$.

\item [Step 3] Solve the hierarchy of \reff{equ_hierarchy_pk}.
If \reff{equ_hierarchy_pk} is infeasible for some order $N$,
then $\lmd_k$ is the smallest eigenvalue and stop.
Otherwise,  we can get the next largest eigenvalue $\lmd_{k+1}$.

\item [Step 4] Let $k:=k+1$ and go to Step~2.

\end{description}
\end{algorithm}

%
%
In Step 2,  if $\chi_k < \lambda_k$,
we should expect $\delta < \lmd_k - \chi_k$.
This is why we update $\dt$ as the minimum of $\dt/5$ and $\lmd_k - \chi_k$.
%
%

\section{Numerical experiments}
\label{section_numerical}

In this section,  we report numerical experiments for showing how to
compute all real eigenvalues. The computation is implemented in a
Thinkpad W520 Laptop,  with an Intel$^\circledR$ dual core CPU at 2.20GHz $\times$ 2
and 8GB of  RAM,  in a Windows~7 operating system.
We use the software Matlab 2013a and {\tt GloptiPoly~3}~\cite{GloPol3}
to solve the semidefinite relaxations for polynomial optimization problems.
In the display of numerical results,  we only show four decimal digits.
%
%

By the definition of $\mtB$-eigenvalues as in \reff{Beig},
$(\lambda,  u)$ is an eigenpair if and only if
$((-1)^{m-m^\prm}\lambda,  -u)$ is an eigenpair.
For H-eigenvalues ($m=m^\prm$),  the H-eigenvectors always
appear in $\pm$ pairs; so we only list H-eigenvectors $u$
satisfying $\Sigma_i {u_i}\ge 0$.
For Z-eigenvalues ($m^\prm=2$),  when $m$ is even,
the Z-eigenvectors appear in $\pm$ pairs,
and  we only list those $u$ satisfying $\Sigma_i {u_i}\ge 0$;
when $m$ is odd, $(\lmd,u)$ is a Z-eigenpair if and only if
$(-\lmd, -u)$ is a Z-eigenpair, and   they appear in $\pm$ pairs.

If the rank condition \reff{cd:flat} is satisfied,
then we can get the $\mtB$-eigenvalue $\lmd_k$ and
$\ell:= \rank\,  M_t(y^*)$ $\mtB$-eigenvectors associated to $\lmd_k$.
When primal-dual interior point methods are applied to solve
the semidefinite relaxations and \reff{cd:flat} holds,
generally all $\mtB$-eigenvectors associated to $\lmd_k$ can be obtained.
We refer to Remarks~\ref{rmk:geomul} and \ref{nb:egv:lmdk}.
In our numerical experiments,  the SDP solver {\tt SeDuMi}~\cite{sedumi}
is called by the software {\tt GloptiPoly~3}. The solver
{\tt SeDuMi} is based on primal-dual interior point methods.
So,  when the rank condition \reff{cd:flat} is satisfied,
we typically get all $\mtB$-eigenvectors of $\lmd_k$.
In such cases, the real geometric multiplicities of computed eigenvalues are also known.
In the display of our numerical results,  we use the notation $\lmd^{(\ell)}$
to mean that $\ell$ distinct $\mtB$-eigenvectors (modulo scaling)
are found for the eigenvalue $\lmd$.

When $\lmd_k$ has infinitely many $\mtB$-eigenvectors on $\mtB x^{m^\prm}=1$,
the rank condition \reff{cd:flat} is typically not satisfied.
%
%
To the best of the authors' knowledge, for such cases,
it is a theoretically open question to check
the convergence of \reff{equ_hierarchy_p1}-\reff{largest_sos}
and \reff{equ_hierarchy_pk}-\reff{largest_r_sos},
although they are proved to have finite convergence
in Theorems~\ref{thm:lmd:1st} and \ref{thm:lmd:k+1}. However,
in practice, this issue can be fixed heuristically as follows.
The sequence $\{ \rho_N^{(k)} \}$
always has finite convergence to $\lmd_k$.
After an approximate convergence of $\rho_N^{(k)}$ is observed,
we can use such $\rho_N^{(k)}$ as an approximation of $\lmd_k$.
Let $\eps>0$ be small such that $\lmd_k$ is a unique $\mtB$-eigenvalue of $\mtA$
in the interval $[\lmd_k-\eps,  \lmd_k+\eps]$.
Choose a generic vector $c \in \mbR^n$ and then solve the problem
\be \label{rand:egvector}
\left\{ \baray{rl}
 \min &  c^{\text{T}}x   \\
  \st &  h_r(x)=0 \,  (r=1, \ldots, 2n-2),   \\
 &  \lmd_k -\eps  \leq \mtA x^m \leq  \lmd_k + \eps.
\earay \right.
\ee
When $c$ is generic, \reff{rand:egvector} has a unique minimizer,
which is a $\mtB$-eigenvector associated to $\lmd_k$.
We can construct semidefinite relaxations, like \reff{equ_hierarchy_pk}-\reff{largest_r_sos},
for solving \reff{rand:egvector}. A $\mtB$-eigenvector
can be found by solving the semidefinite relaxations
(cf.~\cite[\S3]{Nie-ft}).
In practice, a generic $c$
can be chosen as a random vector in $\mbR^n$. In Matlab,
we can set {\tt c = randn(n,1)}.
A small enough $\eps$ can be chosen as follows.
We first assign a small value to $\eps$, say, $0.05$.
After solving \reff{rand:egvector}, we are done
if a $\mtB$-eigenvector associated to $\lmd_k$ is found;
otherwise, update $\eps := \eps/5$ and solve \reff{rand:egvector} again.
Repeat this process, until a $\mtB$-eigenvector $u$ associated
to $\lmd_k$ is found. Once $u$ is obtained,
we check the equation $\mathcal{A} u^{m-1} = \lmd_k \mathcal{B}u^{m^\prime-1}$.
If it is satisfied, then $(\lmd_k,u)$ is confirmed
to be an eigenpair.
In our examples,
we use the superscript $^{(\star)}$ to mean that an
eigenvector is computed by solving \reff{rand:egvector}.
%
%

\begin{example}$($\cite{Qi05}$)$
\label{example4-1}
Consider the tensor $\mtA \in \mathtt{S}^4(\mbR^3)$ such that
\[  \mtA x^4 = x_1^4+2x_2^4+3x_3^4. \]
It is a diagonal tensor (i.e.,  its entries $\mtA_{i_1 i_2 i_3}$ are all zeros except for
$i_1=i_2=i_3$). Its Z-eigenvalues were computed by Qi~\cite[Proposition~9]{Qi05}.
For this tensor,  the optimization problem \reff{align_general_new} is
\begin{align*}
  \max \quad & x_1^4+2x_2^4+3x_3^4 \\
  \st \quad & 2x_1x_2^3-x_2x_1^3=0,  \ 3x_1x_3^3-x_3x_1^3=0,  \\
             & 3x_2x_3^3-2x_3x_2^3=0,  \ x_1^2+x_2^2+x_3^2=1.
\end{align*}
Using Algorithm~\ref{alg:alleig},  we get all the real Z-eigenvalues
and Z-eigenvectors,  which are shown in Table~\ref{table_4-1}.
The computation takes about $9$ seconds.
 \begin{table}[h]
  \centering \small{
   \caption{Z-eigenpairs of the tensor in Example~\ref{example4-1}}
  \begin{tabular}{|c|r@{\hspace{1mm}}|r@{\hspace{1mm}}|@{\hspace{1mm}}r@{\hspace{1mm}}|r@{\hspace{1mm}}|@{\hspace{1mm}}r@{\hspace{1mm}}|@{\hspace{1mm}}r@{\hspace{1mm}}|r@{\hspace{1mm}}|}
   \hline
  $k$ &$\,\,\,\,1$\quad\, &$\,\,\,\,2$\quad\, &$3$\quad\,  &$\,\,\,\,4$\quad\,  &$5$\quad\,  &$\,\,\,\,6$ \quad\,  &$7$\quad\quad \\ \hline
 $\lambda_k$  & $3.0000$   &  $2.0000$  & \ $1.2000^{(2)}$ &   $1.0000$  &  \,\,$0.7500^{(2)}$ &  \, $0.6667^{(2)}$
          & $0.5455^{(4)}$  \\ \hline
    & $0.0000$   &  $0.0000$   &   $0.0000$    &  $1.0000$  &  $0.8660$ &  $0.8165$
    & $\pm 0.7386$      \\
 $u_k$ & $0.0000$  &  $1.0000$   &   $0.7746$  & $0.0000$  &  $0.0000$  & $\pm 0.5773$
               & $\pm 0.5222$      \\
    & $1.0000$  &  $0.0000$  & $\pm 0.6324$   &  $0.0000$  & $\pm 0.5000$
          &  $0.0001$  &  $0.4264$         \\ \hline
  \end{tabular}\label{table_4-1}}
\end{table}
\end{example}

\begin{example}\label{example4-2}
For the diagonal tensor $\mathcal{D}\in \mathtt{S}^5(\mbR^4)$ such that
$\mathcal{D} x^5 = x_1^5 + 2x_2^5-3x_3^5-4x_4^5$,
its orthogonal transformations have the same Z-eigenvalues as
$\mathcal{D}$ (cf.~Qi~\cite[Theorem~7]{Qi05}).
Consider $\mathcal{A}\in \mathtt{S}^5(\mbR^4)$ such that
$\mtA x^5 = \mathcal{D} (Px)^5$, where
\[
P = (I - 2w_1w_1^T)(I - 2w_2w_2^T)(I-2w_3w_3^T)
\]
and $w_1, w_2, w_3$ are randomly generated unit vectors.
%
%
The order $m=5$ is odd,  so the Z-eigenvalues of $\mtA$ appear in $\pm$ pairs. Using Algorithm~\ref{alg:alleig},
we get all 30 real Z-eigenvalues. It takes about $400$ seconds. The nonnegative Z-eigenvalues are
\begin{align*}
4.0000 , \quad  3.0000  , \quad   2.0000  , \quad   1.2163, \quad 1.0000, \quad0.9611,  \quad  0.8543, \quad 0.6057, \\
0.5550, \quad 0.5402, \quad 0.4805, \quad 0.3887, \quad0.3466,   \quad  0.3261,   \quad   0.2518.\quad \quad \quad \quad\,
\end{align*}
For cleanness,  the Z-eigenvectors are not shown.
\end{example}

\begin{example}\label{example4-3} $($\cite[Example~3]{Qi05}$)$
Consider the tensor $\mtA \in \mathtt{S}^4(\mbR^3)$ such that
\[
\mtA x^4 = 2x_1^4+3x_2^4+5x_3^4+4ax_1^2x_2x_3,
\]
where $a$ is a parameter.
The polynomial optimization problem \reff{align_general_new} is
\begin{align*}
  \max \quad & 2x_1^4+3x_2^4+5x_3^4+4ax_1^2x_2x_3 \nn \\
  \st \quad & x_1^{p-1}(3x_2^3+ax_1^2x_3)-x_2^{p-1}(2x_1^3+2ax_1x_2x_3)=0, \nn \\
             & x_1^{p-1}(5x_3^3+ax_1^2x_2)-x_3^{p-1}(2x_1^3+2ax_1x_2x_3)=0, \nn \\
             & x_2^{p-1}(5x_3^3+ax_1^2x_2)-x_3^{p-1}(3x_2^3+ax_1^2x_3)=0, \nn \\
             & x_1^p+x_2^p+x_3^p=1,
\end{align*}
where $p=2$ for Z-eigenvalues and $p=4$ for H-eigenvalues. Using Algorithm~\ref{alg:alleig},
we get all the real Z and H eigenvalues,  which are shown in Table~\ref{table4-3}.
For each value of $a$, it takes a couple of seconds (from $5$ to $20$).
%
%
\begin{table}[h]
\centering
\caption{Z-eigenvalues and H-eigenvalues of the tensor in Example~\ref{example4-3}} \label{table:example1}
\small{
\begin{tabular}{|l|l@{\hspace{2mm}}l@{\hspace{2mm}}l@{\hspace{2mm}}l@{\hspace{2mm}}r@{\hspace{2mm}}
r@{\hspace{2mm}}r@{\hspace{2mm}}r@{\hspace{2mm}}|}
\hline
&\multicolumn{8}{c|}{ \rm Z-eigenvalues }\\ \hline
$a=0$ &  $5.0000$  &  $3.0000$  &  $2.0000$  &   $1.8750^{(2)}$   & $1.4286^{(2)}$  &  $1.2000^{(2)}$ &   $0.9679^{(4)}$  & \\
$a=0.25$& $5.0000$  & $3.0000$  &  $2.0000$  &   $1.8750^{(2)}$  &  $1.4412^{(2)}$   & $1.2150^{(2)}$ &  $1.0881^{(2)}$  &  $0.8646^{(2)}$\\
$a=0.5$& $5.0000$  &  $3.0000$  &  $2.0000$  &  $1.8750^{(2)}$  &  $1.4783^{(2)}$   & $1.2593^{(2)}$  &  $1.2069^{(2)}$  &  $0.7243^{(2)}$\\
$a=1$& $5.0000$  &  $3.0000$  &  $2.0000$  &  $1.8750^{(2)}$  &  $1.6133^{(2)}$   &$0.4787^{(2)}$&   & \\
$a=3$ & $5.0000$ &  $3.0000$   & $2.2147^{(2)}$  &  $2.0000$   & $1.8750^{(2)}$ &  $-0.5126^{(2)}$  &  &\\ \hline\hline
&\multicolumn{8}{c|}{ \rm H-eigenvalues }\\ \hline
$ a=0$ &   $5.0000$  & $3.0000$  & $2.0000$  & &&&&\\
$a=0.25$&  $5.0009^{(2)}$  & $5.0000$  &  $3.0000$   &  $2.0000$   &  $1.9310^{(2)}$   & &&\\
$a= 0.5$&  $5.0137^{(2)}$  & $5.0000$  &  $3.0000$   &  $2.0000$   &  $1.7517^{(2)}$   &  &   & \\
$a= 1 $&   $5.1812^{(2)}$  & $5.0000$  &  $3.0000$   &  $2.0000$   &  $1.2269^{(2)}$   &   & & \\
$ a= 3 $&  $7.4505^{(2)}$  & $5.0000$  &  $3.0000$   &  $2.0000$   &  $-1.3952^{(2)}$  &   & & \\
\hline
\end{tabular} \label{table4-3}
}\end{table}
For cleanness,  the eigenvectors are not shown.
\end{example}

\begin{example}
$($\cite[Example~4]{Qi05}$)$  \label{example4-4}
Let $\mtA \in \mathtt{S}^4(\mbR^2)$ be the tensor such that
\[
\mtA x^4 = 3 x_1^4 + x_2^4 + 6a x_1^2 x_2^2,
\]
where $a$ is a parameter. As shown in~\cite{Qi05},  this tensor always has two
Z-eigenvalues $\lmd=3,\lmd=1$. When $a<\frac{1}{3}$ or $a>1$,
$\mtA$ has another double Z-eigenvalue
\[
\frac{3(9a^3-6a^2-3a+2)}{2(3a-2)^2}.
\]
For some values of $a$,  the Z-eigenvalues are shown in Table \ref{table4-4}.
For each case of $a$,  the computation takes about $1$ second.
%
%
{\small
\begin{table}[h]
\centering
\caption{Z-Eigenvalues of the tensor in Example~\ref{example4-4}}
  \begin{tabular}{|l|l|c|r||l|c|c|c|}
   \hline
   & \quad $\lambda_1$&$\lambda_2$&$\lambda_3$ \quad \qquad &&$\lambda_1$&$\lambda_2$&$\lambda_3$\\ \hline
$a=-1$ & $3.0000$ & $1.0000$  & $-0.6000^{(2)}$ & $a=0$ & $3.0000$ & $1.0000$ & $0.7500^{(2)}$  \\
$a=0.25$ & $3.0000$ & $1.0000$ & $0.9750^{(2)}$ & $a=0.5$   & $3.0000$ & $1.0000$  & \\
$a=2$  & $4.1250^{(2)}$ & $3.0000$ & $1.0000$  \, \,   &&&&\\ \hline
  \end{tabular}\label{table4-4}
\end{table}
}
\end{example}

\begin{example}\label{example4-5}
$($\cite[Example~3.5]{SSHOPM},~\cite[Example 3.4]{nie2013semidefinite}$)$
Consider the tensor $\mtA \in \mathtt{S}^4(\mbR^3)$ such that
{\smaller \smaller
\[
\baray{rcrcr}
\mtA_{1111} = 0.2883,  &\, \mtA_{1112} = -0.0031,  & \mtA_{1113} = 0.1973,  & \ \mtA_{1122} = -0.2485,  &  \mtA_{1123} = -0.2939,  \\
\mtA_{1133} = 0.3847,  & \mtA_{1222} = \,\,0.2972,  &  \mtA_{1223} = 0.1862,  &  \mtA_{1233} =\,\,0.0919,  & \mtA_{1333} = -0.3619,  \\
\mtA_{2222}=0.1241,  &  \mtA_{2223} = -0.3420,  & \mtA_{2233}=0.2127,  & \mtA_{2333} =\,\, 0.2727,  &  \mtA_{3333}=-0.3054.
\earay
\]}Using Algorithm~\ref{alg:alleig},
we get all the real Z-eigenvalues and Z-eigenvectors.
They are shown in Table~\ref{table4-5}. The computation takes about $9$ seconds.
\begin{table}[h]
\centering {\scriptsize
   \caption{Z-eigenpairs  of the tensor in Example~\ref{example4-5}}
  \begin{tabular}{|@{\hspace{0.5mm}}c@{\hspace{0.5mm}}|c@{\hspace{1mm}}|c@{\hspace{1mm}}|r@{\hspace{1mm}}|c
  @{\hspace{1mm}}|r@{\hspace{1mm}}|r@{\hspace{1mm}}|r@{\hspace{1mm}}|r@{\hspace{1mm}}|r
  @{\hspace{1mm}}|r@{\hspace{1mm}}|r@{\hspace{1mm}}|}
   \hline
$k$   & $1 $ & $2$ & $3$ \quad\ & $4 $ & $5 $\ \ \, \qquad &
             $6$\, \,\qquad & $7$\, \, \qquad & $8$\, \, \qquad
         & $9$ \,  \qquad & $10$ \, \, \qquad & $11$ \,\, \qquad\\ \hline
$\lambda_k$& $0.8893$ & $0.8169$ & $0.5105$ & $0.3633$ & $0.2682$ & $0.2628$
           & $0.2433$ & $0.1735$ & $-0.0451$ & $-0.5629$  & $-1.0954$ \\ \hline
& $0.6672$ & $0.2471$ & $-0.7027$ & $0.8412$ & $-0.2635$ & $0.4722$ & $0.3598$
        & $-0.7780$ & $0.5150$ & $0.2676$  &  $0.6447$ \\
$u_k$ & $0.7160$ & $0.6099$ & $0.4362$ & $0.6616$ & $-0.1318$ & $0.4425$ & $0.8870$
          & $0.9895$ & $0.0947$ & $-0.1088$  & $0.3357$ \\
& $0.9073$ & $0.2531$ & $0.7797$ & $0.6135$ & $0.1250$ & $0.1762$ & $-0.1796$
         & $0.9678$ & $-0.5915$ & $0.7467$  & $0.3043$ \\
\hline
 \end{tabular}\label{table4-5}
 }
\end{table}

\end{example}

%
%
%
%

\begin{example}\label{example4-6}
$($\cite[Example 9.1]{Qi09}$)$
%
%
Consider the tensor $\mtA \in \mathtt{S}^3(\mbR^6)$ such that
\[  \mtA x^3  = x_1^3+\cdots+x_6^3 + 30x_1^2x_2+\cdots+30x_5^2x_6. \]
It is a cubic tensor of dimension six. Its Z-eigenvalues
appear in $\pm$ pairs. In total,  there are $19$ nonnegative Z-eigenvalues:
{\small
\[
\baray{rrrrrrrrlr}
16.2345 &  15.4552  &  15.4298 &   10.9710  &  8.7347& 8.6596 &    8.5979&
8.1888  &    7.2165 &   6.0000\\
5.5674 &   5.5668  &  5.5218  &   5.4817   &  5.1402 &   4.3358  &
4.2464  &     4.0225 &    3.9992. &
\earay
\]} It takes about $10870$ seconds to compute them.
%
%
%
For cleanness,  the Z-eigenvectors are not shown.
\end{example}

Characteristic tensors of hypergraphs have important applications,
as shown in Li~et al.~\cite{li2013z}.
The second largest Z-eigenvalue can be used to get a lower bound for the bipartition width.
The following is such an example.

\begin{example}\label{example-hypergraph6}
$($\cite[Example 6.4]{li2013z}$)$
%
%
Consider the tensor $\mtA \in \mathtt{S}^4(\mbR^6)$ such that
\begin{align*}
  - \mtA x^4 =& (x_1-x_2)^4+(x_1-x_3)^4+(x_1-x_4)^4+(x_1-x_5)^4+(x_1-x_6)^4\\
  &+(x_2-x_3)^4+(x_2-x_4)^4+(x_2-x_5)^4+(x_2-x_6)^4\\
 & +(x_3-x_4)^4+(x_3-x_5)^4+(x_3-x_6)^4\\
  &+(x_4-x_5)^4+(x_4-x_6)^4+(x_5-x_6)^4.
\end{align*}
The polynomial $\mtA x^4$ is symmetric in $x$.
Every permutation of a Z-eigenvector is also a Z-eigenvector.
So,  we can add extra conditions $x_1 \geq x_2 \geq \cdots \geq x_6$ to
\reff{align_general_new} and \reff{largest-r-Jacobian},
while not changing eigenvalues. Then we solve the corresponding semidefinite relaxations.
The tensor $\mtA$ has five real Z-eigenvalues,  which are respectively
\[
\lmd_1= 0.0000,  \, \,   \lmd_2=  -4.0000,  \, \,  \lmd_3=  -4.5000,  \, \,
 \lmd_4=  -6.0000,  \, \, \lmd_5=  -7.2000.
\]
The Z-eigenvectors,  whose entries are ordered monotonically decreasing,
are shown in Table~\ref{table4-7}.
It takes about $280$ seconds to get them.
\begin{table}[h]
\centering {\small
  \caption{Z-eigenpairs of the tensor in Example~\ref{example-hypergraph6}}
  \begin{tabular}{|c|l|r@{\hspace{2mm}}r@{\hspace{2mm}}r@{\hspace{2mm}}r@{\hspace{2mm}}r@{\hspace{2mm}}r|}
   \hline
$k $ & \qquad $\lambda_k$  &\multicolumn{6}{c|}{$u_k^{\text{T}}$} \\ \hline
$ 1$& \,\,\, $0.0000$ & $(0.4082$ & $0.4082$ & $0.4082$ & $0.4082$ & $0.4082$ & $0.4082)$\\
$ 2$&$-4.0000^{(20)}$ & $(0.4082$ & $0.4082$ & $0.4082$ & $-0.4082$ & $-0.4082$ & $-0.4082)$\\
$ 3$& $-4.5000^{(\star)}$\,    &  $(0.2887$  &  $0.2887$ &  $0.2887$  & $0.2887$ & $-0.5774$  & $-0.5774)$ \\
$ 4$& $-6.0000^{(15)} $& $(0.7071$ & $0.0000$ & $0.0000$ & $0.0000$ & $0.0000$ & $-0.7071)$\\
$ 5$&$-7.2000^{(6)}$ \quad & $(0.1826$ &  $0.1826$  &  $0.1826$ &  $0.1826$  & $0.1826$ & $-0.9129)$\\
\hline
\end{tabular}\label{table4-7}
}
\end{table}
In the computation of $\lmd_3$,  the rank condition \reff{cd:flat} is not satisfied.
We get one of its Z-eigenvectors by solving \reff{rand:egvector}.
\end{example}

%
%

\begin{example}\label{example4-8}
$($\cite[Example 2]{xie2013z}$)$
%
%
Consider the tensor $\mtA \in \mathtt{S}^4(\mbR^5)$ such that
\[
\mtA x^4 = (x_1+x_2+x_3+x_4)^4+(x_2+x_3+x_4+x_5)^4.
\]
Using Algorithm~\ref{alg:alleig},
we get all the three real Z-eigenvalues of this tensor,
which are respectively
\[
\lmd_1 = 24.5000, \quad \lmd_2 = 0.5000, \quad  \lmd_3 = 0.0000.
\]
It takes about $320$ seconds to get them. The Z-eigenvectors are shown in Table~\ref{table4-8}.
\begin{table}[h]
\centering  {\small
  \caption{ Z-eigenpairs of the tensor in Example~\ref{example4-8} }
  \begin{tabular}{|c|l|r@{\hspace{2mm}}r@{\hspace{2mm}}r@{\hspace{2mm}}r@{\hspace{2mm}}r|}
   \hline
 k& \qquad $\lambda_k$&\multicolumn{5}{c|}{$u_k^{\text{T}}$} \\ \hline
 $1$& $24.5000$ & $(0.2673$ & $0.5345$ & $0.5345$ & $0.5345$ & $0.2673)$  \\
 $2$& \,  $0.5000$  & $(0.7071$& $0.0000$ & $0.0000$ & $0.0000$ & $-0.7071)$   \\
 $3$& \,  $0.0000^{(\star)}$ & $(0.5253$  & $0.3021$  & $-0.4781$  &  $-0.3472$  &  $0.5318)$ \\
\hline
\end{tabular}\label{table4-8}
}
\end{table}
There are infinitely many Z-eigenvectors for $\lmd_3$.
In the computation of $\lmd_3$,  the rank condition \reff{cd:flat} is not satisfied.
So,  we solve \reff{rand:egvector} and get one of its Z-eigenvectors.
\end{example}

\begin{example}$($\cite[Example~5.7]{cartwright2011number}$)$
\label{example4-9NEW}
Consider the cubic tensor $\mtA \in \mathtt{S}^3(\mbR^3)$ such that
\[
\mtA x^3 =  2x_1^3 + 3x_1x_2^2 + 3x_1x_3^2.
\]
Using Algorithm~\ref{alg:alleig}
we get two real Z-eigenvalues,  which are $\lmd_1 = 2$ and $\lmd_2 = -2$.
Their Z-eigenvectors are $(1,  0,  0)$ and $(-1,  0,  0)$ respectively.
It takes about $1$ second to compute them.
\end{example}

\begin{example}
$($\cite[Example~5.8]{cartwright2011number}$)$
\label{example4-11}
Consider the tensor $\mtA \in \mathtt{S}^6(\mbR^3)$ such that
\[
\mtA x^6 = x_1^4x_2^2 + x_1^2x_2^4 + x_3^6 - 3x_1^2x_2^2x_3^2,
\]
which is the Motzkin polynomial.
Since $\mtA x^6$ has only even powers in each of $x_1, x_2, x_3$,
we can add the extra conditions $x_1\geq 0,  x_2\geq 0,  x_3\geq 0$
to \reff{align_general_new} and \reff{largest-r-Jacobian},
while not changing eigenvalues. Then we solve
the corresponding semidefinite relaxations.
The tensor $\mtA$ has three real H-eigenvalues.
Using Algorithm~\ref{alg:alleig},  we get all of them,  which are respectively
\[
\lmd_1 = 1.0000, \,   \lmd_2 = 0.0555, \,  \lmd_3 = 0.0000.
\]
The H-eigenvectors are shown in Table~\ref{table4-11}.
It takes about $30$ seconds.
\begin{table}[h]
\centering
   \caption{H-eigenpairs of the tensor in Example~\ref{example4-11}}
  \begin{tabular}{|c|c|rrr|}
  \hline
$k$ & $\lambda_k$ & \multicolumn{3}{c|}{$u_k^{\text{T}}$}\\ \hline
$1$& $1.0000^{(3)}$& $(0.0000$ & $0.0000$ & $1.0000)$  \\
 &       & $(0.8909$ & $\pm 0.8909$ & $0.0000)$  \\
$2$& $0.0555^{(8)}$& $(0.4487$  &  $\pm 0.9823$  & $\pm 0.6735)$  \\
 &               & $(0.9823$  &  $\pm 0.4487$  & $\pm 0.6735)$ \\
$3$& $0.0000^{(6)}$& $(1.0000$ & $0.0000$ & $0.0000)$  \\
 &             & $(0.0000$ & $1.0000$ & $0.0000)$  \\
  &            & $(0.8327$ & $\pm0.8327$ & $\pm0.8327)$ \\
\hline
 \end{tabular}\label{table4-11}
\end{table}
\end{example}

\begin{example}\label{example4-10}
$($\cite[Example~3.5]{nie2013semidefinite}$)$
Consider the tensor $\mtA \in \mathtt{S}^3(\mbR^n)$ such that
\[
\mtA_{ijk} = \frac{(-1)^i}{i} + \frac{(-1)^j}{j} +  \frac{(-1)^k}{k}.
\]
For the case $n = 5$,  we get all the real Z-eigenvalues,  which are respectively
\[
\lmd_1 = 9.9779,  \ \lmd_2 = 4.2876,  \ \lmd_3 =  0.0000,  \ \lmd_4 =  -4.2876,  \ \lmd_5 = -9.9779.
\]
The computation takes about $150$ seconds.
The Z-eigenvectors of $\lmd_1, \lmd_2, \lmd_3$ are shown in Table~\ref{table4-10}.
The Z-eigenvector of $\lmd_4$ (resp.,  $\lmd_5$) is just the negative
of that of $\lmd_2$ (resp.,  $\lmd_1$).
\begin{table}[h]
\centering
\caption{Z-eigenpairs of the tensor in Example~\ref{example4-10}}
  \begin{tabular}{|c|l|rrrrr|}
  \hline
$k$ & \quad\ $\lambda_k$   & \multicolumn{5}{c|}{$u_k^{\text{T}}$}\\ \hline
$1$& $9.9779$                & $(-0.7313$  & $-0.1375$ &  $-0.4674$  &  $-0.2365$  &  $-0.4146)$  \\
$2$& $4.2876$                &$(-0.1859$  & $0.7158$  &  $0.2149$  &  $0.5655$  &  $0.2950)$  \\
$3$& $0.0000^{(\star)}$   & $(\, \, \, \, 0.5072$  &  $-0.0980$ &  $0.4280$  &  $-0.7344$  & $-0.1028)$ \\
\hline
 \end{tabular}\label{table4-10}
\end{table}
In the computation of $\lmd_3$,  the rank condition \reff{cd:flat} is not satisfied.
We get a Z-eigenvector for $\lmd_3$ by solving \reff{rand:egvector}.
\end{example}

\begin{example}
$($\cite{nie2013semidefinite}$)$ \label{example4-12}
Consider the tensor $\mtA \in \mathtt{S}^4(\mbR^n)$ such that
\[
\mtA_{i_1 \ldots  i_4} = \sin(i_1+i_2+i_3+i_4).
\]
For the case $n=5$,  we get all the real Z-eigenvalues which are respectively
\[
\lmd_1 = 7.2595,  \ \lmd_2 =  4.6408,   \ \lmd_3 = 0.0000,  \ \lmd_4 =  -3.9204,  \ \lmd_5 =  -8.8463.
\]
The Z-eigenvectors are shown in Table~\ref{table4-12}.
It takes about $370$ seconds.
\begin{table}[h]
\centering  \small
\caption{Z-eigenpairs  of the tensor in Example~\ref{example4-12}}
  \begin{tabular}{|c|c|rrrrr|}
  \hline
$ k$ &  \quad $\lambda_k$   & \multicolumn{5}{c|}{$u_k^{\text{T}}$}\\ \hline
$ 1$ & \, \,    $7.2595$  &  $(\, \, \, \, 0.2686$ &  $0.6150$  &  $0.3959$  &  $-0.1872$  &  $-0.5982)$ \\
$ 2$ & \, \,  $4.6408$ & $(-0.5055$ & $0.1228$  & $0.6382$  & $0.5669$ & $-0.0256)$  \\
$ 3$ & \quad\, \,  $0.0000^{(\star)}$ &  $(\, \, \, \, 0.5935$ & $0.3675$  &  $-0.1224$ &  $0.5449$  &  $0.4341)$  \\
$ 4$ & $-3.9204$ & $(-0.1785$  &  $0.4847$  &  $0.7023$  &  $0.2742$  &  $-0.4060)$   \\
$ 5$ & $-8.8463$ & $(-0.5809$ &  $-0.3563$  &  $0.1959$  &  $0.5680$  &  $0.4179)$  \\
\hline
 \end{tabular}\label{table4-12}
\end{table}
In the computation of $\lmd_3$,  the rank condition \reff{cd:flat} is not satisfied.
We get a Z-eigenvector for $\lmd_3$ by solving \reff{rand:egvector}.
\end{example}

\begin{example}\label{example4-13}
Consider the tensor $\mtA \in \mathtt{S}^4(\mbR^n)$ such that
\[
\mtA_{i_1 \ldots  i_4} = \tan(i_1) +\tan(i_2) + \tan (i_3) + \tan (i_4).
\]
For the case $n=5$,  we get all the real Z-eigenvalues which are respectively
\[
\lmd_1 =  34.5304,  \quad  \lmd_2 = 0.0000,  \quad \lmd_3=-101.1994.
\]
The Z-eigenvectors are displayed in Table~\ref{table4-13}.
It takes about $170$ seconds to compute them.
\begin{table}[h]
\centering \small
   \caption{Z-eigenpairs  of the tensor in Example~\ref{example4-13}}
  \begin{tabular}{|c|r|rrrrr|} \hline
 $ k$ & $\lambda_k$ \quad \, &  \multicolumn{5}{c|}{$u_k^{\text{T}}$} \\ \hline
 $ 1$ &  $34.5304$ \,\,\,\, &      $(\,\,\,\,0.6665$ & $0.1089$ & $0.4132$ &  $0.6070$ & $-0.0692)$ \\
 $ 2$ &$0.0000^{(\star)}$ &       $(-0.7276$ & $-0.1080$ & $0.4238$ & $0.5178$ & $-0.1060)$  \\
 $ 3$ & $-101.1994 \,\,\,\,\,\,$   &  $(\,\,\,\,0.2248$ & $0.5541$ & $0.3744$ & $0.2600$ & $0.6953)$ \\
\hline
 \end{tabular}\label{table4-13}
\end{table}
In the computation of $\lmd_2$,  the rank condition \reff{cd:flat} is not satisfied.
We get a Z-eigenvector for $\lmd_2$ by solving \reff{rand:egvector}.
\end{example}

\begin{example}\label{example4-14}
Consider the tensor $\mtA \in \mathtt{S}^5(\mbR^n)$ such that
\[
\mtA_{i_1 \ldots  i_5} =  \ln(i_1) + \cdots +  \ln(i_5).
\]
For the case $n=4$,  we get all the real Z-eigenvalues which are respectively
\[
\lmd_1 = 132.3070, \, \,  \lmd_2 =    0.7074,  \, \,  \lmd_3 =  0.0000, \,  \,
\lmd_4 =  -0.7074,  \, \,  \lmd_5 =  -132.3070.
\]
The Z-eigenvectors of $\lmd_1, \lmd_2, \lmd_3$ are shown in Table~\ref{table4-14}.
The Z-eigenvector of $\lmd_4$ (resp.,  $\lmd_5$) is just the negative
of that of $\lmd_2$ (resp.,  $\lmd_1$).
It takes about $420$ seconds to compute them.
\begin{table}[h]
\centering \small
   \caption{Z-eigenpairs  of the tensor in Example~\ref{example4-14}}
  \begin{tabular}{|c|r|rrrr|} \hline
$k$ &$ \lambda_k$    \, \, \, \, \,  & \multicolumn{4}{c|}{$u_k^{\text{T}}$}\\ \hline
$1$& $132.3070$    \,\,  &   $(\, \, \, \, 0.4030$ &   $0.4844$ &   $0.5319$  &  $0.5657)$ \\
$2$&   $\,\,\,\,0.7074$  \,\,\,  &  $(-0.9054$ &  $-0.3082$ &   $0.0411$  &  $0.2890)$ \\
$3$&   $\,\,\,\,0.0000^{(\star)}$ &  $(-0.8543$ &  $0.2645$  &  $0.4168$  &  $0.1617)$ \\
\hline
 \end{tabular}\label{table4-14}
\end{table}
In the computation of $\lmd_3$,  the rank condition \reff{cd:flat} is not satisfied.
We get a Z-eigenvector for $\lmd_3$ by solving \reff{rand:egvector}.
\end{example}

\begin{example} \label{nbeig:rand}
({\bf random tensors})
An interesting question is to determine 
the number of real Z-eigenvalues for the symmetric tensors.
Cartwright and Sturmfels~\cite[Theorem~5.5]{cartwright2011number} showed that
every symmetric tensor $\mtA$ of order $m$ and dimension $n$ has at most
\[
M(m, n) \, := \, \frac{(m-1)^n-1}{m-2}
\]
distinct complex Z-eigenvalues.
In \cite{cartwright2011number},
$(\lambda, u)$ and $((-1)^m\lambda, -u)$ are considered to be the same eigenpair.
To be consistent with \cite{cartwright2011number},
for odd ordered tensors, we here only count their nonnegative Z-eigenvalues.
Furthermore, they also showed that when $\mtA$ is generic,
$\mtA$ has exactly $M(m,n)$ distinct complex Z-eigenvalues.
Clearly, $M(m,n)$ is an upper bound for the number of real Z-eigenvalues.
But it might not be sharp for generic tensors.
In this example,  we explore possibilities of distributions of
the numbers of real Z-eigenvalues.
For each $(m, n)$,  we generate $50$ symmetric tensors randomly.
Each symmetric tensor is generated as the symmetrization of
a random nonsymmetric tensor ${\tt randn(i_1, \ldots, i_m)}$ in Matlab.
The number of their real Z-eigenvalues are shown in Table~\ref{table:randexample}.
The notation $k^{\{\mu\}}$ means that
there are $\mu$ instances for which the number of real Z-eigenvalues equals to $k$.
%
%
\begin{table}[h]
\centering \small
\caption{Numbers of real Z-eigenvalues of random symmetric tensors}\label{table:randexample}
\label{table_random}
\begin{tabular}{|@{\hspace{1mm}}c|@{\hspace{1mm}}c|l@{\hspace{2mm}}|} \hline
$(m,n)$ & \tiny{$M(m, n)$} & \qquad \qquad
                   {\rm Numbers of real $Z$-eigenvalues}  \\ \hline
$(3,5)$ & $31$  &  $7, 9^{\{5\}}, 11^{\{7\}}, 13^{\{4\}},
                   15^{\{6\}}, 17^{\{2\}},19^{\{17\}}, 21^{\{7\}}, 23$ \\ \hline
$(3,4)$ & $15$  &  $3^{\{5\}}, 5^{\{8\}}, 7^{\{8\}}, 9^{\{11\}},
                   11^{\{17\}}, 13$ \\  \hline
$(3,3)$ & $7$   &  $1^{\{4\}}, 3^{\{17\}}, 5^{\{ 16\}}, 7^{\{ 13\}}$ \\ \hline
$(4,4)$ & $40$  &  $8^{\{2\}}, 10^{\{3\}}, 12^{\{8\}}, 14^{\{6\}}, 16^{\{8\}},
                   18^{(12)},  20^{\{6\}}, 22^{\{3\}},  24, 28$ \\  \hline
$(4,3)$ & $13$  &  $3^{\{3\}}, 5^{\{5\}}, 7^{\{27\}}, 9^{\{8\}},
                   11^{\{7\}} $  \\ \hline
$(5,4)$ & $85$  &  $13,15^{\{4\}}, 17^{\{8\}}, 19^{\{2\}}, 21^{\{8\}}, 23^{\{8\}},
                   25^{\{7\}}, 27^{\{9\}},   31^{\{2\}}, 33 $ \\ \hline
$(5,3)$ & $21$  &  $5^{\{2\}}, 7^{\{9\}}, 9^{\{13\}}, 11^{\{17\}},
                   13^{\{7\}}, 15^{\{2\}}$  \\ \hline
\end{tabular}
 \end{table}
The table confirms that $M(m,n)$ is an upper bound for the numbers of real Z-eigenvalues.
Moreover, the numbers of real Z-eigenvalues are not evenly distributed.
We do not know the reason for such distributions.
\end{example}

Theoretically, Algorithm \ref{alg:alleig} is able to compute all real eigenvalues
for all symmetric tensors, provided that the computer has sufficient capacity.
In practice, the sizes of symmetric tensors,
for which the eigenvalues can be computed by Algorithm \ref{alg:alleig},
depend on the computer memory and the relaxation order $N$.
The length of the variable $y$ in \reff{equ_hierarchy_p1}
and \reff{equ_hierarchy_pk} is $\binom{n+2N}{2N}$.
It grows fast in the order $N$. In our computational experiences,
for general tensors, a small order $N$ is often enough.
This fact was observed for random tensors in Example~\ref{nbeig:rand}.
However, for some special tensors, a big order $N$ might be required. For such cases, it is often very hard to compute all real eigenvalues.

\bigskip

A different approach for computing all real eigenvalues
is based on solving the system \reff{Beig} directly for its real solutions.
This can be done by using the numerical solver {\tt NSolve}
provided by {\tt Mathematica}.
Generally, {\tt NSolve} can solve relatively small problems.
The following is such an example.

\begin{example} \label{compare:Nsolve}
Consider the symmetric tensor $\mtA \in \mathtt{S}^4(\mbR^n)$ such that
\begin{align*}
  \mtA x^4 =& (x_1-x_2)^4+\cdots+(x_1-x_{n})^4+(x_2-x_3)^4+\cdots+(x_2-x_{n})^4\\
  &+\cdots+(x_{n-1}-x_n)^4.
\end{align*}
Like in Example \ref{example-hypergraph6},
we can add extra conditions $x_1 \geq x_2 \geq \cdots \geq x_n$ to
\reff{align_general_new} and \reff{largest-r-Jacobian},
while not changing eigenvalues.
We compute its real Z-eigenvalues.
The computational results are shown in Table \ref{table4-17}.
\begin{table}[h]
\centering \small
   \caption{Z-eigenpairs  of the tensor in Example~\ref{compare:Nsolve}}
  \begin{tabular}{|c|c|r|lllll|} \hline
 & $alg.$& \text{time(s)} & \multicolumn{5}{c|}{Z-eigenvalues}\\ \hline
 \multirow{2}[0]{*}{$n=4$} &{\rm Alg.\ \ref{alg:alleig}} & $\ \ 3.6$ &$5.3333$&$5.0000$&$4.0000$&$0.0000$&\\
                  &{\tt NSolve}& $\ 19.3$&$5.3333$&$5.0000$&$4.0000$&$0.0000$&\\
\hline
 \multirow{2}[0]{*}{$n=5$} &{\rm Alg.\ \ref{alg:alleig}}  & $274.5$ &$6.2500$&$5.5000$&$4.2500$&$4.1667$&$0.0000$\\
                  &{\tt NSolve}& $-$&$-$&$-$&$-$&$-$ &$-$\\
\hline
\multirow{2}[0]{*}{  $n=6$ } &{\rm Alg.\ \ref{alg:alleig}}  & $280.2$ &$7.2000$&$6.0000$&$4.5000$&$4.0000$&$0.0000$\\
                  &{\tt NSolve}& $-$&$-$&$-$&$-$&$-$ &$-$\\
\hline
\multirow{2}[0]{*}{ $n=7$ } &{\rm Alg.\ \ref{alg:alleig}}  & $9565.6$ &$ 8.1667^{(2)}$&$6.5000$&$4.9000^{(2)}$&$4.8846^{(2)}$&$4.7500$ \\
 &&&$4.1667$&$4.0883^{(2)}$&$0.0000$&&\\
                  &{\tt NSolve}& $-$&$-$&$-$&$-$&$-$ &$-$\\
\hline
\multirow{2}[0]{*}{ $n=8$ } &{\rm Alg.\ \ref{alg:alleig}}  & $938.2$ &$ 9.1429^{(2)}$&$7.0000$&$5.3333^{(2)}$&$-$&$-$ \\
                  &{\tt NSolve}& $-$&$-$&$-$&$-$&$-$ &$-$\\
\hline
\multirow{2}[0]{*}{ $n=9$ } &{\rm Alg.\ \ref{alg:alleig}}  & $4173.8$ &$ 10.1250^{(2)}$&$7.5000$&$5.7857^{(2)}$&$-$ & $-$ \\
                  &{\tt NSolve}& $-$&$-$&$-$&$-$&$-$ &$-$\\
\hline
\multirow{2}[0]{*}{ $n=10$ } &{\rm Alg.\ \ref{alg:alleig}}  & $15310.5$ &$ 11.1111^{(2)}$&$8.0000$&$6.2500^{(2)}$&$-$&$-$ \\
                  &{\tt NSolve}& $-$&$-$&$-$&$-$&$-$ &$-$\\
\hline
 \end{tabular}\label{table4-17}
\end{table}
For the case $n=4$, Algorithm \ref{alg:alleig} takes about $3$ seconds,
while {\tt NSolve} takes about $19$ seconds.
They both get all the real Z-eigenvalues correctly.
For the case $n=5$, Algorithm \ref{alg:alleig} gets all the real Z-eigenvalues
in about $274$ seconds, while {\tt NSolve} can't get answers in $5$ hours
(we terminated the computation after $5$ hours).
In Table \ref{table4-17}, ``$-$" means that no computational results are returned.
We can also get all real eigenvalues for $n=6,7$.
For the bigger $n = 8, 9, 10$, we can get the first three largest Z-eigenvalues,
but the other smaller Z-eigenvalues cannot be obtained.
This is because, for such cases, we need to use the relaxation order $N=4$,
which causes the computer to run out of memory.
%
%
For $n = 8, 9, 10$, the reported time is only
for the first three biggest Z-eigenvalues.
For the values of $n$ bigger than $10$, the computer runs out of memory
and we cannot get the eigenvalues.
%
%
\end{example}

In Algorithm~\ref{alg:alleig}, if the real eigenvalues are not
separated well, then the positive number $\dt>0$
need to be chosen very small.
%
%
We consider the following
example, thanks to an anonymous referee.

\begin{example}
\label{example4-17}
Consider the tensor $\mtA \in \mathtt{S}^3(\mbR^2)$ such that
\[
\mtA_{111} = 1,\quad \mtA_{222} = 1+10^{-6},
\]
and all the other entries are zeros. In Algorithm \ref{alg:alleig},
to get the real Z-eigenvalues correctly, the value of $\dt$
decreased to be smaller than $10^{-6}$ during the loop.
The computed nonnegative real Z-eigenvalues are
\[
\lmd_1 = 1.000001,  \quad  \lmd_2 = 1.000000,  \quad \lmd_3=0.707107.
\]
The whole computation takes about $2$ seconds.
\end{example}

%
%

We conclude this section by exploring how Algorithm \ref{alg:alleig} scales
in terms of sizes of tensors.

\begin{example} \label{scale:rand}
We explore the sizes of symmetric tensors for which Algorithm \ref{alg:alleig}
can get all their Z-eigenvalues. Randomly generated symmetric tensors are tested,
in the same way as in Example~\ref{nbeig:rand}.
The dimensions and orders of random symmetric tensors,
whose real Z-eigenvalues can be found by Algorithm \ref{alg:alleig},
are shown in Table \ref{table_scaling}.
In the left half of Table \ref{table_scaling},
we choose values $n=3,4,5,6$. For each $n$ of them,
we list the values of $m>2$ such that
we can find all real Z-eigenvalues by Algorithm \ref{alg:alleig}.
Similarly, in the right half of Table \ref{table_scaling},
we choose values $m=3,4,5,6$. For each $m$ of them,
we list the values of $n$ such that
we can find all real Z-eigenvalues by Algorithm \ref{alg:alleig}.
\end{example}
\begin{table}[h]
\centering \small
\caption{Scaling of Algorithm \ref{alg:alleig} for computing all the Z-eigenvalues}
\label{table_scaling}
 \begin{tabular}{||c|cccccccc||c|cccccc||}     \hline
 $n$& \multicolumn{8}{c||}{$m$} &$m$& \multicolumn{6}{c||}{$n$}\\ \hline
 3 & 3 & 4 & 5 & 6& 7 &8&9&10  &   3 & 2 & 3 & 4 & 5 & 6 &7 \\
 4 & 3 & 4 & 5 &6 & &  && &     4 & 2 & 3 & 4 & 5 & 6 & \\
 5 & 3 &  4 & 5 &  & &   && &      5 & 2 & 3 & 4 & 5  &  & \\
 6 & 3 & 4 &  &  & &   &&  &       6 & 2 & 3 & 4 &  &  & \\
 \hline
 \end{tabular}
 \end{table}

\bigskip
\noindent {\bf Acknowledgement} \,
Chun-Feng Cui was partially supported by the Chinese NSF Grant (no. 11301016).
Yu-Hong Dai was partially
supported by the Chinese NSF grants (nos. 11331012 and 81173633) and the China
National Funds for Distinguished Young Scientists (no. 11125107). Jiawang Nie was partially
supported by the NSF grants DMS-0844775, DMS-1417985.


\begin{thebibliography}{100}

\bibitem{Bruns}
{\sc W.~Bruns and U.~Vetter}. {\em Determinantal rings},  Lecture Notes in Mathematicas  1327,  Springer-Verlag, 1988.

\bibitem{cartwright2011number}
{\sc D.~Cartwright and B.~Sturmfels}. {\em The number of eigenvalues of a
  tensor},  Linear Algebra and its Applications,   438~(2013),  pp.~942--952.


\bibitem{chang2009eigenvalue}
{\sc K.~C. Chang,  K.~Pearson, and T.~Zhang}. {\em Perron-Frobenius theorem for
nonnegative tensors},  Communications in Mathematical Sciences,
6~(2008),  pp.~507--520.



\bibitem{chen2013positive}
{\sc Y.~Chen, Y.-H.~Dai, D.~Han, and W.~Sun}.
{\em Positive semidefinite generalized diffusion tensor imaging
via quadratic semidefinite programming},
SIAM Journal on Imaging Sciences, 6~(2013), pp.~1531--1552.



\bibitem{cox2007ideals}
{\sc D.~Cox,  J.~Little, and D.~O'Shea.}
{\em Ideals,  varieties,  and algorithms: An Introduction to
computational algebraic geometry and commutative algebra},  Springer,  2007.


\bibitem{CF05}
{\sc R.~Curto and L.~Fialkow}.
{\em Truncated $K$-moment problems in several variables},
Journal of Operator Theory, 54~(2005),  pp. 189-226.


\bibitem{Demmel}
{\sc J.~Demmel}.
{\em Applied Numerical Linear Algebra},
Society for Industrial and Applied Mathematics, 1997.

%
%


\bibitem{Han12}
{\sc L.~Han}. {\em An unconstrained optimization approach for finding real
  eigenvalues of even order symmetric tensors},
Numerical Algebra,  Control and Optimization, 3~(2013),  pp.~583-599.




\bibitem{Hao12}
{\sc C.~Hao,   C.~Cui,  and Y.-H.~Dai}. {\em A sequential subspace projection
  method for extreme Z-eigenvalues of supersymmetric tensors},  Numerical Linear Algebra with Applications, to appear, 2014.

%
%

\bibitem{HenLas05}
{\sc D.~Henrion and J.~B.~Lasserre}.
{\em Detecting global optimality and extracting solutions in GloptiPoly},
Positive polynomials in control,  Lecture Notes in Control and Information Science,
Springer,  Berlin, 312~(2005), pp.~293--310.




\bibitem{GloPol3}
{\sc D.~Henrion,  J.~B.~Lasserre, and J.~Loefberg}.
{\em GloptiPoly~3: moments,  optimization and semidefinite programming},
Optimization Methods and Software, Vol.~24, No. 4-5, pp.~761--779, 2009.


%
%


\bibitem{hillar2009most}
{\sc C.~Hillar and L.-H.~Lim}. {\em Most tensor problems are NP-hard},
Journal of the ACM,  60~(2013),  no. 6.


\bibitem{HHQ13}
{\sc S.~Hu,  Z.~H.~Huang, and L.~Qi}.
{\em Finding the extreme Z-eigenvalues of tensors via a sequential
semidefinite programming method},
Numerical Linear Algebra with Applications,  20~(2013),  pp.~972--984.



\bibitem{Kolda2009}
{\sc T.~G.~Kolda and B.~W.~Bader}.  {\em Tensor decompositions and
  applications},  SIAM Review,  51~(2009),  pp.~455--500.

\bibitem{SSHOPM}
{\sc T.~G.~Kolda and J.~R.~Mayo}. {\em Shifted power method for computing
  tensor eigenpairs},  SIAM Journal on Matrix Analysis and Applications,  32~(2011),  pp.~1095--1124.

\bibitem{lasserre2001global}
{\sc J.~B.~Lasserre}. {\em Global optimization with polynomials and the problem
  of moments},  SIAM Journal on Optimization,  11~(2001),  pp.~796--817.


\bibitem{LasBok}
{\sc J.~B.~Lasserre}.
{\em Moments,  positive polynomials and their applications},
Imperial College Press,  2009.


\bibitem{Lau}
{\sc M.~Laurent}.
{\em Sums of squares,  moment matrices and optimization over polynomials},
Emerging Applications of Algebraic Geometry,  IMA Volumes in Mathematics and its Applications
(Eds. M. Putinar and S. Sullivant),  Springer, 149~(2009),  pp.~157--270.

%


\bibitem{li2013z}
{\sc G.~Li,  L.~Qi,   and G.~Yu}. {\em The Z-eigenvalues of a symmetric tensor and
  its application to spectral hypergraph theory},  Numerical Linear Algebra with
  Applications,  20~(2013),  pp.~1001--1029.


\bibitem{Lim05}
{\sc L.-H.~Lim}. {\em Singular values and eigenvalues of tensors: a variational
  approach},   Proceedings of the IEEE International Workshop
  on Computational Advances in Multi-Sensor Adaptive Processing (CAMSAP '05),
  1~(2005),  pp.~129--132.

\bibitem{Lim13}
{\sc L.-H.~Lim}. {\em Tensors and hypermatrices},  L. Hogben (Ed.), Handbook of Linear Algebra, 2nd Ed., CRC Press, Boca Raton, FL, 2013.


\bibitem{ni2008eigenvalue}
{\sc Q.~Ni,  L.~Qi,   and F.~Wang}.  {\em An eigenvalue method for testing positive
  definiteness of a multivariate form},  Automatic Control,  IEEE Transactions
  on,  53~(2008),  pp.~1096--1107.


\bibitem{nia2007degree}
{\sc G.~Ni,  L.~Qi,  F.~Wang,  and Y.~Wang}. {\em The degree of the
  E-characteristic polynomial of an even order tensor},
  Journal of Mathematical Analysis and Applications,
  329~(2007),  pp.~1218--1229.


\bibitem{nie2013exact}
{\sc J.~Nie}. {\em An exact Jacobian SDP relaxation for polynomial
  optimization},  Mathematical Programming,  137~(2013),  pp.~225--255.

\bibitem{Nie-ft}
{\sc J.~Nie}. {\em Certifying convergence of Lasserre's hierarchy via flat truncation},
Mathematical Programming,  Ser. A, 142~(2013), no.~1-2,  pp.~485--510.



\bibitem{nie2013hierarchy}
{\sc J.~Nie}. {\em The hierarchy of
  local minimums in polynomial optimization}, 
Mathematical Programming, to appear. 


\bibitem{nie2013semidefinite}
{\sc J.~Nie and L.~Wang}.  {\em Semidefinite relaxations for best rank-1 tensor
approximations},
SIAM Journal on Matrix Analysis and Applications,
Vol.~35, No.~3, pp.~1155--1179, 2014.



\bibitem{Qi05}
{\sc L.~Qi}.  {\em Eigenvalues of a real supersymmetric tensor},  Journal of
  Symbolic Computation,  40~(2005),  pp.~1302--1324.



\bibitem{QiSunWang07}
{\sc L.~Qi,  W.~Sun,   and Y.~Wang}. {\em Numerical multilinear algebra and its
  applications},  Frontiers of Mathematics in China,  2~(2007),  pp.~501--526.


\bibitem{Qi03}
{\sc L.~Qi and K.~L.~Teo}. {\em Multivariate polynomial minimization and its
  application in signal processing},  Journal of Global Optimization,  26~(2003),
  pp.~419--433.

\bibitem{Qi09}
{\sc L.~Qi,  F.~Wang,   and Y.~Wang}.  {\em Z-eigenvalue methods for a global
  polynomial optimization problem},  Mathematical Programming,  118~(2009),
  pp.~301--316.


\bibitem{Qi08}
{\sc L.~Qi,  Y.~Wang,   and E.~X.~Wu}. {\em D-eigenvalues of diffusion kurtosis
  tensors},  Journal of Computational and Applied Mathematics,  221~(2008),
  pp.~150--157.

\bibitem{Qi10}
{\sc L.~Qi,  G.~Yu,  and E.~X.~Wu}. {\em Higher order positive semidefinite
  diffusion tensor imaging},  SIAM Journal on Imaging Sciences,  3~(2010),
  pp.~416--433.

\bibitem{Rez00}
{\sc B.~Reznick}.
{\em Some concrete aspects of Hilbert's $17^{th}$ problem},
Contemp. Math., American Mathematical Society, 253~(2000),  pp.~251--272.
%


\bibitem{sedumi}
{\sc J.~F.~Sturm}.
{\em SeDuMi 1.02: a MATLAB toolbox for optimization over symmetric cones},
 Optimization Methods and Software,  11\&12~(1999), pp.~625--653.
\url{http://sedumi.ie.lehigh.edu/.}


%
%


\bibitem{SDPbook}
{\sc H.~Wolkowicz, R.~Saigal, and L.~Vandenberghe (Eds.)}.
{\em Handbook of Semidefinite Programming:
Theory, Algorithms, and Applications},
International Series in Operations Research \& Management Science,
Springer, 2000.



\bibitem{xie2013z}
{\sc J.~Xie and A.~Chang}.  {\em On the Z-eigenvalues of the signless Laplacian
  tensor for an even uniform hypergraph},  Numerical Linear Algebra with
  Applications,  20~(2013),  pp.~1030--1045.


\bibitem{ZLQ12}
{\sc X.~Zhang,  C.~Ling, and L.~Qi}.
{\em The best rank-1 approximation of a symmetric tensor and related
spherical optimization problems},
SIAM Journal on Matrix Analysis and Applications,
33~(2012), no.~3, pp.~806--821.


\end{thebibliography}
\end{document}